\documentclass[11pt]{amsart}

\usepackage{amsthm,amsfonts, amssymb, amscd}
\usepackage{youngtab}
\usepackage{ytableau}
\usepackage[all]{xy}
\usepackage{euscript}
\usepackage{cite}

\usepackage{tikz}
\usetikzlibrary{matrix,arrows, decorations.markings}
\usetikzlibrary{positioning}

\usepackage[normalem]{ulem}

\usepackage{hyperref}
\hypersetup{%
    linktoc=page
}

\let\oldtocsection=\tocsection
\let\oldtocsubsection=\tocsubsection
\renewcommand{\tocsection}[2]{\hspace{0em}\oldtocsection{#1}{#2}}
\renewcommand{\tocsubsection}[2]{\hspace{1em}\oldtocsubsection{#1}{#2}}

\newtheorem{thm}{Theorem}[section]
\newtheorem{lemma}[thm]{Lemma}
\newtheorem{cor}[thm]{Corollary}
\newtheorem{prop}[thm]{Proposition}

\theoremstyle{remark}
\newtheorem{Rem}[thm]{Remark}
\theoremstyle{remark}
\newtheorem{example}[thm]{Example}
\theoremstyle{definition}

\newtheorem{innercustomprop}{Proposition}
\newenvironment{customprop}[1]
{\renewcommand\theinnercustomprop{#1}\innercustomprop}
{\endinnercustomprop}

\theoremstyle{definition}

\theoremstyle{definition}

\newtheorem{definition}[thm]{Definition}

\numberwithin{equation}{section}

\newcommand{\C}{\mathbb{C}}           
\newcommand{\N}{\mathbb{ N}}           
\newcommand{\Z}{\mathbb{ Z}}           
 
\newcommand{\Q}{\mathbb{ Q}}           
\newcommand{\Ad}{\operatorname{Ad }}             

\newcommand{\Hom}{\operatorname{Hom}}

\newcommand{\Ind}{\operatorname{Ind}}

\newcommand{\fb}{{\mathfrak b}}

\newcommand{\fg}{{\mathfrak g}}

\newcommand{\fl}{{\mathfrak l}}

\newcommand{\fp}{{\mathfrak p}}

\newcommand{\ft}{{\mathfrak t}}
\newcommand{\fu}{{\mathfrak u}}
\newcommand{\fs}{{\mathfrak s}}

 \newcommand{\cb}{\mathcal{B}}
\newcommand{\cc}{\mathcal{C}}

\newcommand{\ch}{\mathcal{H}}
 
 \newcommand{\cl}{\mathcal{L}}
 \newcommand{\cm}{\mathcal{M}}
 \newcommand{\cn}{\mathcal{N}}
 \newcommand{\co}{\mathcal{O}}
 \newcommand{\cp}{\mathcal{P}}
  
 \newcommand{\cs}{\mathcal{S}}
 \newcommand{\ct}{\mathcal{T}} 
   
 \newcommand{\cv}{\mathcal{V}}
 \newcommand{\cx}{\mathcal{X}}

\newcommand{\hess}[2]{\mathcal{B}_{#1}^{#2}}
\newcommand{\ideals}{\mathfrak Id}
\newcommand{\cosets}{{W^0}}
\newcommand{\Ar}{{\mathrm{Irr}(A(x))}}
\newcommand{\Irr}{{\mathrm{Irr}}}
\newcommand{\Wr}{{\Irr(W)}}
\newcommand{\pairs}{(\co, \cl)}
\newcommand{\bpairs}{\co, \cl}  

\newcommand{\allpairs}{\Theta}  
\newcommand{\sppairs}{\Theta_{sp}}  

\newcommand{\Ch}{\mathsf{C}\mathrm{h}}

\renewcommand{\tilde}{\widetilde}

\newcommand{\sh}{\underline{\C}}

\renewcommand{\bar}[1]{\overline{#1}}

\newcommand{\map}{\Psi}
\newcommand{\LLT}{\mathrm{LLT}}

\begin{document}


\title[Perverse sheaves, nilpotent Hessenberg varieties, and the modular law]{Perverse sheaves, nilpotent Hessenberg varieties, and the modular law}

\author{Martha Precup}
\address{Department of Mathematics and Statistics\\ Washington University in St. Louis \\ One Brookings Drive\\ St. Louis, Missouri\\ 63130\\ USA  }
\email{martha.precup@wustl.edu}

\author{Eric Sommers}
\address{Department of Mathematics and Statistics\\ University of Massachusetts Amherst \\Lederle Tower \\ Amherst, MA \\ 01003 \\ USA  }
\email{esommers@math.umass.edu}

\dedicatory{Dedicated to George Lusztig}

\date{\today}

\maketitle

\begin{abstract}
We consider generalizations of the Springer resolution of the nilpotent cone of a simple Lie algebra by replacing the cotangent bundle with certain other vector bundles over the flag variety.  We show that the analogue of the Springer sheaf has as direct summands only intersection cohomology sheaves that arise in the Springer correspondence.   The fibers of these general maps are nilpotent Hessenberg varieties, and we build on techniques established by De Concini, Lusztig, and Procesi to study their geometry.  For example, we show that these fibers have vanishing cohomology in odd degrees. This leads to several implications for the dual picture, where we consider maps that generalize the Grothendieck--Springer resolution of the whole Lie algebra.  
In particular we are able to prove a conjecture of Brosnan.
	
As we vary the maps, the cohomology of the corresponding nilpotent Hessenberg varieties often satisfy a relation we call the geometric modular law, which also has origins in the work of De Concini, Lusztig, and Procesi.  We connect this relation in type $A$ with a combinatorial modular law defined by Guay-Paquet that is satisfied by certain symmetric functions and deduce some consequences of that connection.
\end{abstract}


\section{Introduction}
Let $G$ be a  simple algebraic group over $\C$ with Lie algebra $\fg$.  Let $B$ be a Borel subgroup with Lie algebra $\fb$ containing a maximal torus 
$T$ with Lie algebra $\ft$.  Denote by $\Phi$ the root system associated to the pair $(T,B)$, with simple roots $\Delta$.  Let $U$ be the unipotent radical of $B$ with $\fu$ its Lie algebra. 
Let $W=N_G(T)/T$ be the Weyl group of $T$ and $\cb:= G/B$ the flag variety of $G$.

Let $\cn$ denote the variety of nilpotent elements in $\fg$.  
For complex varieties, $\dim(X)$ refers to the complex dimension.  
Let $N: = \dim (\cn)$, which equals $2 \dim(\fu)$.
The Springer resolution of the nilpotent cone $\cn$ in $\fg$ is the proper, $G$-equivariant 
map 
$$\mu: G \times^B \fu \to \cn$$ 
sending $(g,x) \in G \times \fu$ to $g.x$, 
where $g.x := \Ad(g)(x)$ denotes the adjoint action of $G$ on $\fg$.

Let $\sh[N]$ be the shifted constant sheaf on $G\times^B\fu$ with coefficients in $\C$.  The shift makes it  a $G$-equivariant
perverse sheaf on $G \times^B \fu$.
A central object in Springer theory is
the Springer sheaf $R\mu_* (\sh[N])$, the derived pushforward of $\sh$ under $\mu$.
The Springer sheaf is a $G$-equivariant perverse sheaf on $\cn$.
The nilpotent cone $\cn$ is stratified by nilpotent $G$-orbits.   Let $\co$ be a nilpotent orbit and $\cl$ 
an irreducible $G$-equivariant local system on $\co$.   Denote by $\allpairs$ the set of all such pairs $(\co,\cl)$.
Let $IC(\bar{\co}, \cl)$ denote the intersection cohomology sheaf on $\cn$ defined by a pair $(\co,\cl) \in \allpairs$.
We use the convention that if $\co' \subsetneq \co$, then $\ch^{j} IC(\bar{\co},\cl)|_{\co'} = 0$ unless 
$$-\!\dim \co \leq j  < -\!\dim \co'.$$ 
The decomposition theorem implies that $R\mu_*(\sh[N])$ is a direct sum of shifted IC-complexes.  That is,
\begin{eqnarray}\label{eqn.Springer.sheaf1}
	R\mu_*(\sh[N]) \simeq \bigoplus_{(\co, \cl)\in \allpairs} IC(\bar{\co}, \cl)\otimes V_{\co, \cl},
\end{eqnarray}
where each $V_{\co, \cl}$ is a graded complex vector space.  
Since $\mu$ is semismall,  $V_{\co, \cl}$ is concentrated in degree $0$.
Now, both sides of \eqref{eqn.Springer.sheaf1} carry an action of $W$ that makes   the nonzero vector space $V_{\co, \cl}$ into an irreducible representation of $W$.    
Let $\sppairs$  denote the pairs $(\co, \cl)$ for which $V_{\co, \cl} \neq 0$. The Springer correspondence says that  the map
\begin{eqnarray}\label{eqn.Springer.correspondence}
(\co, \cl) \in \sppairs \to V_{\co, \cl} \in \Wr
\end{eqnarray}
is a bijection. Here, $\Irr(K)$ 
denotes the irreducible complex representations of a group $K$.   
Our convention for the Springer correspondence sends the zero orbit with trivial local system to the sign 
representation of $W$ and the regular nilpotent orbit with trivial local system to the trivial representation of $W$.
See~\cite[Chapter 8]{Achar-book} for a more detailed discussion of the Springer correspondence.  

This paper is concerned with 
the generalization of \eqref{eqn.Springer.sheaf1} when $\fu$ is replaced by a subspace $I \subset \fu$ that is $B$-stable, as well as the 
connection of this map to related objects in Lie theory and combinatorics.
The $B$-stable subspaces are also called ad-nilpotent ideals of $\fb$ and are well-studied Lie-theoretic objects (see, for example, \cite{Kostant1998, Cellini-Papi2000}).  
Denote by $\ideals$ the set of all $B$-stable subspaces of $\fu$.  
The cardinality of $\ideals$  is the $W$-Catalan number 
$\prod _{i=1}^n \frac{d_i+h}{d_i}$ where $d_1, \dots, d_n$ are the fundamental degrees of $W$ and $h$ is the Coxeter number.
In type $A$ these ideals are in bijection with Dyck paths (see \S \ref{sec.modular.revisited}).

If $I \in \ideals$, then $G.I$ is the closure of a nilpotent orbit, denoted by $\co_I$.  
The restriction of $\mu$ gives a map
\begin{equation}\label{main_map}
\mu^I: G \times^B I \to \bar{\co_I}
\end{equation}
that is still proper, but it is no longer a resolution or semismall in general. 
Set $N_I: = \dim(G \times^B I )$, which equals $\dim I+\dim G/B$.  
The decomposition theorem still applies to the analogue of the Springer sheaf.  Namely,
\begin{eqnarray}\label{eqn.gen_Springer.sheaf2_intro}
	R\mu^I_*(\sh[N_I]) \simeq \bigoplus_{\pairs \in \allpairs} IC(\bar{\co}, \cl) \otimes V^I_{\co, \cl},
\end{eqnarray}
where $V^I_{\co, \cl}$ is a graded vector space, but no longer concentrated in degree $0$ in general.

Our first main result is that if a pair $(\co, \cl) \in \allpairs$ 
contributes a nonzero term in \eqref{eqn.gen_Springer.sheaf2_intro}, then it must appear in the Springer correspondence, i.e., it contributes a nonzero term in \eqref{eqn.Springer.sheaf1}. 
In other words, 
\begin{thm} \label{ThmA}
Let $I \in \ideals$.  If $V^I_{\bpairs} \neq 0$ in  \eqref{eqn.gen_Springer.sheaf2_intro}, then  $(\co, \cl) \in \sppairs$. 
\end{thm}
The  case when $I  = \fu_P$, the nilradical of the Lie algebra of a parabolic subgroup $P$ of $G$, 
was established by Borho and MacPherson \cite{Borho-MacPherson1983} 
and they gave a formula for the dimensions of the vector spaces $V^I_{\bpairs}$ (see  \eqref{parabolic_case}).

Theorem \ref{ThmA} is proved by analyzing the fibers of the map $\mu^I$.  Let $x \in \cn$, and let
$$\hess{x}{I}:= (\mu^I)^{-\!{1}}(x).$$   
For the $I = \fu$ case, the fibers $\mu^{-1}(x)$ are the Springer fibers and denoted more simply
as $\cb_x$. 
The fiber $\hess{x}{I}$ is a subvariety of $\cb_x$, 
and any variety defined in this way is called a nilpotent Hessenberg variety.

For $x \in \cn$, denote by $\co_x$ the $G$-orbit of $x$ under the adjoint action.  The component group 
$A(x):= Z_G(x)/Z_G^\circ(x)$ 
is a finite group, which identifies with the fundamental group of $\co_x$ when $G$ is simply-connected.
The cohomology of $\hess{x}{I}$ carries an action of $A(x)$ and Theorem \ref{ThmA} is equivalent,
using proper base change, to showing that if $\chi \in \Ar$ has nonzero multiplicity in $H^*(\hess{x}{I})$, then the
pair $(\co_x, \cl_\chi)$ belongs to $\sppairs$.   Here, $\cl_\chi$ denotes the  irreducible  $G$-equivariant local system on $\co_x$
defined by $\chi$.

The analysis of the $\hess{x}{I}$ 
occurs in \S\ref{sec.decomposition}, 
where we establish a decomposition of $\hess{x}{I}$ into vector bundles over a small set of smooth varieties, 
from which we also deduce that $\hess{x}{I}$ has no odd cohomology.
These results are generalizations of those for the Springer fibers $\cb_x$, handled by De Concini, Lusztig and Procesi in \cite{dCLP1988}, and we rely on the techniques developed in that paper.
Theorem~\ref{ThmA} is then proved in \S\ref{sec.proof_thmA}.

Theorem~\ref{ThmA} has an important implication for certain generalizations  of the Grothendieck--Springer resolution.
For $I \in \ideals$, we can consider $I^\perp \subset \fg$, the orthogonal complement to $I$ under the Killing form.  
Then $H=I^{\perp}$ is also $B$-stable and it contains $\fb$, the Lie algebra of $B$.  
The map $\mu^H$  given by
\begin{equation}\label{hessy_main_map}
	\mu^H: G \times^B H \to \fg
\end{equation}
is surjective and generalizes the Grothendieck--Springer resolution for $H = \fb$.
Using Theorem \ref{ThmA} and the Fourier transform, 
we deduce in Theorem \ref{thm.sheaves.ss}, that 	
$$R\mu^H_* (\sh_{H}[\dim G \times^B H])$$ has full support, proving a
conjecture of Brosnan \cite{Xue2020, Vilonen-Xue2021}.  
This generalizes results of  B\u{a}libanu--Crooks  \cite{Balibanu-Crooks2020} who proved the theorem in type $A$ and Xue \cite{Xue2020} who has given a proof in type~$G_2$.
The remainder of \S\ref{sec.gen_grothendieck} studies applications of Theorem \ref{thm.sheaves.ss}.
In Proposition \ref{prop:nilpo_hessy}  we establish a generalization to all types of an unpublished result of Tymoczko and MacPherson in type $A$.  
We conclude \S \ref{sec.gen_grothendieck} by introducing two graded $W$-representations, the dot action representation of Tymoczko
and LLT representations of Procesi and Guay-Paquet.

In \S \ref{sec.modular_law} the main result relates the cohomology of $\hess{x}{I}$ 
for certain triples of subspaces $I \in \ideals$.
The setup is a triple 
$I_2 \subset I_1 \subset I_0$ of ideals in $\ideals$, each of codimension one in the next.
For a simple root $\alpha \in \Delta$, let $P_\alpha$ denote the minimal parabolic subgroup containing $B$ associated to $\alpha$.
The triples $I_2 \subset I_1 \subset I_0$ of interest are those satisfying the following conditions:
\begin{enumerate}
\item $I_2$ and $I_0$ are $P_\alpha$-stable for some $\alpha \in \Delta$, and 
\item 
the representation of the Levi subgroup of $P_\alpha$ (which is of type $A_1$) on the two-dimensional space $I_0/I_2$ is irreducible.
\end{enumerate}
Such triples were introduced in \cite[\S2.8]{dCLP1988}.  
See Definition \ref{basic_move_defn} for a purely root-theoretic definition.
The first example of such a triple occurs when 
$\fg = \mathfrak{sl}_3(\C)$.  Let $\alpha, \beta$ denote the simple roots and let $I_0 = \fu_{P_\alpha}$ be the nilradical of parabolic subalgebra $\mbox{Lie}(P_\alpha)$.  Set $I_2 = \{0\}$.  There is a unique $I_1 \in \ideals$ with $I_2 \subsetneq I_1 \subsetneq I_0$ 
and these three spaces form such a triple.

Our result below is a generalization of  \cite[Lemma 2.11]{dCLP1988} in that it applies to all nilpotent elements, not just to those $x \in \co_{I_0}$.
\begin{prop}[The geometric modular law] \label{geom_modular_fiber}
	Given a triple $I_2 \subset I_1 \subset I_0$ as above and $x \in \cn$, there is an $A(x)$-equivariant isomorphism
	\begin{equation}
		H^j(\hess{x}{I_1}) 	\oplus H^{j-2}(\hess{x}{I_1}) \simeq H^j(\hess{x}{I_0}) \oplus H^{j-2}(\hess{x}{I_2})
	\end{equation}
	for all $j \in \Z$.
\end{prop}

The proposition has a formulation as a statement about perverse sheaves.  
If $Q$ is any parabolic subgroup stabilizing $I$, then we can also 
consider the map $\mu^{I,Q}: G \times ^Q I \to \overline{\co_I}$ and its derived pushforward
$$\cs_{I,Q}:= R\mu^{I,Q}_*(\sh[\dim(G \times ^Q I)]).$$
Proposition \ref{geom_modular_fiber} has the following consequence (see Proposition \ref{geom_modular_v2}): 
for any triple $I_2 \subset I_1 \subset I_0$ as above,
there is an isomorphism
$$\cs_{I_1,B} \simeq \cs_{I_2,P_\alpha} \oplus \cs_{I_0,P_\alpha} $$
in the derived category of $G$-equivariant perverse sheaves on $\cn$.

Proposition \ref{geom_modular_fiber} implies that various polynomials $F^I \in \N[q]$, depending on $I \in  \ideals$,
that arise in our study satisfy 
the following law for a triple of ideals:
$$(1+q) F^{I_1}  =  F^{I_2} + q F^{I_0}.$$
These include the Poincare polynomials of nilpotent Hessenberg varieties and coefficients in the
decomposition of the dot action and LLT representations, see Proposition \ref{mod_laws}.

In type $A$ this law is closely related to a linear relation, called the modular law, satisfied by certain graded symmetric functions.  It is due to Guay-Paquet \cite{Guay-Paquet2013} and studied more recently by Abreu--Nigro \cite{Abreu-Nigro2020}.   It was seeing this law in the combinatorial setting that led us to 
connect it with the work of \cite{dCLP1988}.
Indeed, we show in \S\ref{sec.modular.revisited} that the geometric modular law of 
Proposition \ref{geom_modular_fiber} implies the combinatorial modular law.  
This allows us to give another proof 
of the Shareshian and Wachs Conjecture \cite[Conjecture 1.4]{Shareshian-Wachs2016}
and to show that the Frobenius characteristic of the LLT representation in type $A$ is a unicellular LLT polynomial, see
 Corollary~\ref{cor.typeA-characters}.  The key idea, due to Abreu--Nigro \cite{Abreu-Nigro2020},
 is that in type $A$ any set of polynomials $F^I$ for $I \in \ideals$ 
 satisfying the modular law are completely determined by the $F^{\fu_P}$ where 
 $P$ is a parabolic subgroup.

\subsection{Acknowledgments}  The authors are grateful to Pramod Achar, Tom Braden, Patrick Brosnan, Timothy Chow, Frank Cole, George Lusztig, Alejandro Morales, and John Shareshian for helpful conversations. 

It is a pleasure to dedicate this paper to George Lusztig, a great advisor and friend, who created so many of the beautiful structures on which this paper rests.

\section{Preliminaries}\label{sec.preliminaries}
Let $\Phi^+, \Phi^-$ and $\Delta$ denote the positive, negative and simple roots associated to the pair  $(T,B)$.  
For a simple root $\alpha \in \Delta$, let $s_\alpha \in W$ denote the corresponding simple reflection.  Let $\ell(w)$ denote the minimal length of $w\in W$
when written as a product of simple reflections.
We fix a representative $w\in N_G(T)$ for each $w\in W$, and denote both by the same letter.
Let $\fg_\gamma  \subset \fg$ denote the root space corresponding to $\gamma\in \Phi$.

If $P$ is a parabolic subgroup of $G$, then $\fp$ denotes its Lie algebra and $\fu_P$ the nilradical of $\fp$.
For $P=B$, we instead use $\fb$ and $\fu$. Generally, $P$ will denote a standard parabolic subgroup, i.e., $B \subset P$.

For a rational representation $M$ of $P$, the smooth variety $G \times^P M$ consists 
of equivalence classes of pairs $(g,m) \in G \times M$ with $(gp,p^{-1}.m) \sim (g,m)$.
If $M \subset \fg$, there is a proper map from $G \times^P M$ to $\fg$ given by $(g,m) \to g.m$.
See~\cite{Jantzen}.

We use $H_*(-)$ for Borel-Moore homology with complex coefficients and $H^*(-)$ for 
singular cohomology with complex coefficients.

\subsection{Grading induced by a nilpotent orbit}

Let $x\in \fg$ be a nonzero nilpotent element 
and recall $\co_x$ is the $G$-orbit of $x$.
By the Jacobson--Morozov theorem, 
$x$ can be completed to a 
$\mathfrak{sl}_2$-triple $\{x,h,y\}\subseteq \fg$.  Namely, there exists $h,y\in \fg$ such that
\[
[h,x]=2x,\quad [h,y] = -2y,\quad [x,y]=h,
\]
which implies $\mathrm{span}_\C\{x,h,y\} \simeq \mathfrak{sl}_2(\C)$.   For $j \in \Z$, let
\[
\fg_j:= \{z\in \fg \mid [h,z] = jz \}.
\]
Without loss of generality, we may conjugate the triple so that $h \in \ft$ and $\alpha(h) \geq 0$ for all $\alpha \in \Delta$.  
Then $h$ and the resulting grading $\fg=\oplus_{i\in \Z} \fg_i$ are uniquely determined by $\co_x$. 
We then have $x \in \fg_2$ and $\fb \subseteq \fp$
where $\fp=\bigoplus_{i\geq 0} \fg_i$ is a parabolic subalgebra of $\fg$.

Let  $P$ and $G_0$ be the connected subgroups of $G$ whose Lie algebras are $\fp$ and $\fg_0$, respectively.  
Let $U_P$ be the unipotent radical of $P$.  The Lie algebra of $U_P$ is $\fu_P =\bigoplus_{i \geq 1} \fg_i$. 
Then  $P =G_0U_P$ is a Levi decomposition of $P$ corresponding to the decomposition $\fp = \fg_0 \oplus \fu_P$. 

Set $B_0= B \cap G_0$, which is a Borel subgroup in $G_0$, with Lie algebra $\fb_0 = \fb \cap \fg_0$.
A key fact is that $\fg_{\geq 2}$ is a $P$-prehomogeneous space, meaning there is a unique dense $P$-orbit.
Indeed, $\co_x \cap \fg_{\geq 2}$ is a $P$-orbit in $\fg_{\geq 2}$ and it is dense.  
Moreover, $\fg_2$ is a $G_0$-prehomogeneous space, with dense orbit $\co_x \cap \fg_2$.
See \cite{Carter1993} for these results.

Since $\ft \subset \fg_0$, 
we can define $\Phi_0 \subseteq \Phi$ to be the roots of $\fg_0$ relative to $\ft$ with simple roots $\Delta_0 := \Delta \cap \Phi_0$ and $\Phi_0^{\pm}:= \Phi^{\pm} \cap \Phi_0$.  
Given a nonzero integer $m$, define $\Phi_m:=\{\gamma\in \Phi \mid \fg_\gamma \subseteq \fg_m\}$ 
and $\Phi_{\geq m} := \cup_{i\geq m} \Phi_i$.

\subsection{$P$-orbits on $\cb$} \label{P_orbits}
 Let $\cb_0$ denote the flag variety $G_0/B_0$. 
The Weyl group of $\fg_0$ is $W_0:=\left< s_\alpha\mid \alpha\in \Delta_0\right>$. 
Let $\cosets$ be the set of right coset representatives for $W_0$ in $W$ of shortest length in their respective cosets.  Then
\begin{equation} \label{minimal_coset_reps}
	\cosets = \{w \in W \mid w^{-1}( \Phi_0^+ ) \subset  \Phi^+ \}.
\end{equation}
The set $\cosets$ parametrizes the $P$-orbits on the $\cb=G/B$.  For $w \in \cosets$,
the corresponding $P$-orbit is $\cp_w:= PwB/B$.  

Let $\lambda: \C^* \to T$ be a co-character satisfying $\alpha(\lambda(z)) = z^{\alpha(h)}$ for $z \in \C^*$ and $\alpha \in \Phi$.  
This gives a $\C^*$-action on $\cb$ preserving each $P$-orbit.  
The fixed points of this $\C^*$-action on $\cp_w$ is $G_0wB/B$, which is isomorphic to $\cb_0$. 
Moreover, the smooth variety $\cp_w$ is a vector bundle over its $\C^*$-fixed points $\cp_w^{\C^*} \simeq \cb_0$
with map $\pi_w: \cp_w \to \cb_0$ given by 
\begin{eqnarray} \label{eqn.bundle}
\pi_w(pwB)  = \lim_{z\to 0} \lambda(z)pwB.
\end{eqnarray}
The fibers of this vector bundle identify with the affine space 
$$\mathbf A^{\ell(w)} \simeq \bigoplus_{\beta \in \Phi^+\cap w(\Phi^-)} \fg_\beta.$$ 
By \eqref{minimal_coset_reps} all roots $\beta$ such that $\beta \in \Phi^+\cap w(\Phi^-)$ belong to $\Phi_{\geq 1}$.  
Hence $\C^*$ acts linearly with positive eigenvalues on the fibers
of this vector bundle, a key fact used in \cite{dCLP1988} to decompose the the Springer fiber $\cb_x$.

\section{A decomposition of Hessenberg varieties} \label{sec.decomposition}

\subsection{Definition of Hessenberg varieties}

The fiber of the map $\mu_I$ in \eqref{main_map} over $x \in \cn$ is given by 
$$\hess{x}{I} = \{gB\in \cb \mid g^{-1}. x\in I\},$$
called a \textbf{nilpotent Hessenberg variety}.
These varieties generalize Springer fibers by replacing $\fu$ in the definition of $\cb_x$ with $I \in \ideals$.  

More generally, let $M$ be a subspace of $\fg$ that is $B$-stable.  Since
$M$ is also $T$-stable, it is a sum of weight spaces of $T$.  Let $\Phi_M$
denote the nonzero weights (i.e., roots) of $T$ that appear in the sum.
For any such $M$ and $x \in \fg$, the Hessenberg variety associated to $x$ and $M$ is the 
closed subvariety $\hess{x}{M}$ of the flag variety defined by 
\begin{equation} \label{definition.varieties}
	\hess{x}{M} = \{gB\in \cb \mid g^{-1}.x\in M\}.
\end{equation}
When $x$ is nilpotent, the $\hess{x}{M}$ are called nilpotent Hessenberg varieties.
We will mainly deal with the case where $x$ is nilpotent, but in \S\ref{monodromy} below, the case where $x$ is regular semisimple also arises.
 
We are interested in two kinds of subspaces of $\fg$ that are stable under the action of $B$.  
Those of the first kind are contained in $\fu$ and 
are called ad-nilpotent ideals (since they are Lie algebra ideals in $\fb$) or just ideals.
Those of the second kind contain $\fb$,
and are called Hessenberg spaces.
Let $\ideals$ denote the set of subspaces of the first kind and $\ch$, those of the second kind.  That is,
$$\ideals= \{I \mid I \subset \fu \text{ and } B.I = I\} \text{ and } \mathcal H= \{H  \mid \fb \subset H \text{ and }  B.H = H\}.$$
The two sets are in bijection.  If $I \in \ideals$, then the orthogonal subspace $I^\perp$ to $I$ under the Killing form
is $B$-stable since $I$ is $B$-stable 
and $I^\perp$ contains $\fb$.  Hence $I^\perp \in \mathcal H$.  Since the Killing form is non-degenerate, 
we have $(I^\perp)^\perp = I$, proving that taking the orthogonal complement defines a bijection between ideals in $\ideals$ and Hessenberg spaces~$\mathcal H$.

Let $P$ be a parabolic subgroup of $G$ with Lie algebra $\fp$.  The varieties $\hess{x}{I}$ for $I \in \ideals$ are also a kind of generalization of Spaltenstein varieties:  
if $I=\fu_P$ is the nilradical of $\fp$, 
then the image of $\hess{x}{I}$ in $G/P$ is the Spaltenstein variety ${\mathcal P}^0_x$
from \cite{Borho-MacPherson1983}.
The varieties $\hess{x}{H}$ for $H \in \mathcal H$ are a kind of generalization of Steinberg varieties:
if $H = \fp$, the image of
$\hess{x}{H}$ in $G/P$ is the Steinberg variety~${\mathcal P}_x$ from \cite{Borho-MacPherson1983}.

We now describe a decomposition of $\hess{x}{M}$ when $x$ is nilpotent that generalizes 
the decomposition (for the Springer fiber $\cb_x$) defined and studied by De Concini, Lusztig, and Procesi in~\cite{dCLP1988}. 
The story from {\it loc.~cit.}~goes through:  $\hess{x}{M}$ decomposes as a union of smooth varieties, each of which is a vector
bundle over one of a small set of smooth varieties.   

For the rest of the section, we fix $x$ nilpotent and its induced grading on $\fg$ as in \S\ref{sec.preliminaries}.

\subsection{Building block varieties}\label{building_blocks}

Recall that $\fg_2$ is a prehomogeneous space for $G_0$, with dense $G_0$-orbit 
$\co_x \cap \fg_2$. Let  $\ideals_{2}$ denote the set of $B_0$-stable linear subspaces of $\fg_2$, and  $\ideals^{gen}_{2} \subset \ideals_{2}$ denote those $U \in \ideals_{2}$ with 
$U  \cap \co_x \neq \emptyset$.  

Following ~\cite[\S 2.1]{dCLP1988},  for $U \in \ideals^{gen}_{2}$ define subvarieties of $\cb_0$ as follows
$$X_U := \{ g B_0 \in G_0/B_0 \mid g^{-1}.x \in U\}.$$
These are smooth, projective varieties and
\begin{equation}\label{dim_building}
\dim X_U = \dim(\cb_0) - \dim(\fg_2/U).
\end{equation}
For example, if $U = \fg_2$ then $X_U = \cb_0$.  
The variety $X_U$ is empty for subspaces $U \in \ideals_2 \setminus \ideals_2^{gen}$.  

\begin{Rem} Let $U\in \ideals_2^{gen}$. If we set $I:= U \oplus \fg_{\geq3}$, then $I \in \ideals$ and 
	$X_U \simeq \hess{x}{I}$ by \cite[Proposition 4.2]{FennThesis}, so the $X_U$ are special cases 
	of the varieties being considered.  
\end{Rem}

\subsection{The decomposition}
Let $M$ be a $B$-stable subspace of $\fg$.
The main result of this section is that the Hessenberg variety $\hess{x}{M}$ decomposes as a 
union of vector bundles over disjoint copies of various $X_U$ for $U \in \ideals^{gen}_{2}$.

\begin{lemma}\label{lemma.coset-ideals} 
For each $w \in \cosets$, the subspace $w.M \cap \fg_2$ of $\fg_2$ is $B_0$-stable. 
\end{lemma}
\begin{proof}  Since $M$ is $B$-stable and hence $T$-stable, $w.M$ is also $T$-stable since  $T=wTw^{-1}$.  
Also being $T$-stable, $w.M$ is a sum of weight spaces for $T$ and thus $w.M \cap\fg_2$ is a sum of root spaces.
Let $\fg_\beta \subset w.M$ and $\fg_\gamma \subset \fb_0$.  Then $\gamma \in \Phi^+_0$ and $w^{-1}(\gamma) \in \Phi^+$ by 
 \eqref{minimal_coset_reps}.
 Now $w^{-1}(\beta) \in \Phi_M$ and $M$ is $B$-stable, so $w^{-1}(\beta) + w^{-1}(\gamma) \in \Phi_M$ if the sum is a root.
 If so,  $w^{-1}(\beta + \gamma) \in \Phi_M$, which means $\beta+\gamma \in \Phi_{w.M}$.  This shows $w.M$ is $B_0$-stable.
 The result follows since $\fg_2$ is $B_0$-stable, being $G_0$-stable.
\end{proof}

Recall that for $w \in \cosets$, $\cp_w$ denotes the $P$-orbit on $\cb$ containing $wB$.  
The next proposition shows that the intersection $\cp_w \cap \hess{x}{M}$ is smooth and describes its structure. 
The proof is a generalization of the methods in ~\cite{dCLP1988}.  
Some cases of these generalizations have previously appeared in~\cite{Precup2013, Fresse2016, Xue2020}.

\begin{prop}\label{prop.P-orbits} Let $w \in \cosets$ and $U = w.M \cap \fg_2 \in \ideals_2$.
	Set $\hess{x, w}{M} = \cp_w \cap \hess{x}{M}$. 
		\begin{enumerate} 
	\item $\hess{x, w}{M} \neq \emptyset$ if and only if $U \in \ideals^{gen}_{2}$.
	
	\item 	If $\hess{x, w}{M}$ is nonempty, then 
	\begin{enumerate}
		\item $\hess{x, w}{M}$ is smooth and 
		$$\dim(\hess{x, w}{M})  =  \ell(w)  + |\Phi_0^+|- |\{ \gamma \in \Phi_{\geq2} \mid w^{-1}(\gamma) \notin \Phi_M \}|.$$ 
		\item $\hess{x, w}{M}$ is a vector bundle over $X_U$ with 
		 dimension of its fiber equal to 
		 $$\ell(w) - |\{\gamma\in \Phi_{\geq3} \mid w^{-1}(\gamma) \notin \Phi_M \}|.$$  
			\end{enumerate}
			\end{enumerate}
\end{prop}

\begin{proof}  
	First, we prove 2(a).
	By definition \eqref{definition.varieties}, 
	$$\hess{x, w}{M} = \left\{ pwB, p \in P  \mid w^{-1}p^{-1}.x \in M \right\}.$$
	Since $p^{-1}.x  \in \fg_{\geq 2}$ for $p \in P$, this can be rewritten as
	$$ \hess{x, w}{M} =\left\{ pwB \mid p^{-1}.x \in w.M\cap \fg_{\geq2} \right\}.$$
	
     The setup in \cite[\S 2.1]{dCLP1988} now applies to  
     the $P$-prehomogeneous space $\fg_{\geq 2}$, 
     the linear subspace of $\fg_{\geq 2}$ equal to $w.M \cap \fg_{\geq 2}$, and the closed subgroup of $P$ equal to 
     $B_w:= P \cap wBw^{-1}$.  
     The subspace  $w.M \cap \fg_{\geq 2}$ is  $B_w$-stable since $M$ is $B$-stable
     and $\fg_{\geq 2}$ is $P$-stable. 
 	Then $\hess{x, w}{M}$ is isomorphic to the subvariety of $P/B_w$ given by
 	$$\{ pB_w \mid p^{-1}.x \in w.M \cap \fg_{\geq2} \},$$
 	which is smooth by Lemma 2.2 in {\it loc. cit.},
 	and its dimension equals
	  $$\dim(P/B_w) -\dim ( \fg_{\geq2}/w.M \cap \fg_{\geq2}).$$
     	Simplifying the dimension formula above using the fact that $\dim(P/B_w) = |\Phi_0^+|+\ell(w)$, 
     	we get
     \begin{eqnarray}\label{eqn.dim1}
     	\dim (\hess{x, w}{M} ) =  |\Phi_0^+|+\ell(w) - |\{ \gamma \in \Phi_{\geq2} \mid w^{-1}(\gamma) \notin \Phi_M\}|.
     \end{eqnarray}
 This completes the proof of 2(a).

   Recall the cocharacter $\lambda$ from \S \ref{P_orbits}.   
   Since $\lambda(z) \cdot x= z^2x$,
   the smooth closed subvariety $\hess{x, w}{M}$ of $\cp_w$ is $\C^*$-stable. 
     Now we can use the result in~\cite[\S 1.5]{dCLP1988} (see also \cite[Theorem 1.9]{Bass-Haboush85}):
     let $\pi_w: \cp_w \to G_0wB$ be the vector bundle map from \S \ref{P_orbits}. 
     Since the $\C^*$-action on $\cp_w$ preserves the fibers of $\pi_w$ and acts with strictly positive weights, 
     it follows that $$\hess{x, w}{M}  \to G_0wB \cap \hess{x, w}{M}$$ is a vector sub-bundle of 
     $\pi_w$.
     Finally, $$G_0wB \cap \hess{x, w}{M} \simeq  \{ g B_0 \in G_0/B_0 \mid g^{-1}.x \in w.M \cap \fg_2\} = X_U$$
     as in  \S 3.7 of {\it loc. \!cit.}
	
By \eqref{dim_building},
\begin{eqnarray*}
\dim X_{U} &=& \dim (G_0/B_0) - \dim( \fg_2/U)\\
&=& |\Phi_0^+| - |\{ \gamma \in \Phi_2 \mid w^{-1}(\gamma) \notin \Phi_{M} \}|.
\end{eqnarray*}
Subtracting this value from the one in~\eqref{eqn.dim1} completes the proof of 2(b).
	
Finally, the proof of (1).   If $\hess{x, w}{M}$ is nonempty, then by 2(b) it follows that 
$X_U$ is nonempty,  so it contains some $gB_0$ for $g \in G_0$, i.e.,  $g^{-1}. x \in U$.  
Hence the $G$-orbit of $x$ meets $U$ and $U \in \ideals_{2}^{gen}.$
Conversely, if $U \in \ideals_{2}^{gen}$ then $g^{-1}. x \in U$ for some $g \in G_0$, which means $gB \in \hess{x, w}{M}$.
\end{proof}

The centralizer $Z_G(x)$ acts on $\hess{x}{M}$ and this gives
an action of the component group $A(x) = Z_G(x)/Z^\circ_G(x)$ 
on the Borel-Moore homology $H_*(\hess{x}{M})$ of $\hess{x}{M}$ 
since the induced action of a connected group is trivial.

There is also an action of $A(x)$ on $H_*(\hess{x, w}{M})$ and $H_*(X_U)$.  
Namely, $Z_G(x) = Z_P(x)$  \cite[Proposition 5.9]{Jantzen} and $Z_P(x)$ acts on $\hess{x, w}{M}$, so $A(x)$ acts on $H_*(\hess{x, w}{M})$.
Set  $L = G_0$, then the centralizer $Z_L(x)$ acts on $X_U$.  Since $Z_L(x)$  is a Levi factor of $Z_P(x)$, then $A(x) \simeq Z_L(x)/Z_L^\circ(x)$ acts on $H_*(X_U)$. 

For each $w\in \cosets$, let $t_M(w) = \ell(w) - |\{\gamma\in \Phi_{\geq3} \mid w^{-1}(\gamma) \notin \Phi_M \}|$,
the dimension of the fiber of the vector bundle from Proposition~\ref{prop.P-orbits}.
The following corollary is implicit in \cite{dCLP1988}.

\begin{cor} \label{relate_cohom}
Let $w \in \cosets$ and $U = w.M \cap \fg_2$.  
Let $U \in  \ideals^{gen}_{2}$.  
Then  as $A(x)$-modules, we have the isomorphism
\begin{equation}
H_{j+2t_M(w)}(\hess{x, w}{M}) \simeq H_{j} (X_{U})
\end{equation}
for all $j \in \mathbb N$.
\end{cor} 
\begin{proof}
The isomorphism as vector spaces follows from the vector bundle result in Proposition 2(a).  Now the action of $Z_P(x)$ and $Z_L(x)$
on $\hess{x, w}{M}$ induce the same action of $A(x)$.  Since the $Z_L(x)$-action commutes with the $\C^*$-action coming from $\lambda$,
the action of $\ell \in Z_L(x)$ commutes with the map $\pi_w$ defined in \eqref{eqn.bundle} and the result follows.
\end{proof}

We also need the following crucial result from \cite{dCLP1988}.  
This is proved by a reduction to distinguished nilpotent orbits, 
where the classical cases are handled by explicit computation and the exceptional groups are handled by a method that we review in \S\ref{sec.modular_law}.

\begin{thm}[Theorem 3.9 in \cite{dCLP1988}] \label{no_odd_dclp}
Let $U \in  \ideals^{gen}_{2}$.  
Then $H_i(X_U)=0$ for $i$ odd.
\end{thm}

For $U \in \ideals^{gen}_2$, define a subset $W_{M,U}$ of $\cosets$ and polynomial $g_{M,U}(q)$ by 
\begin{eqnarray}\label{eqn.WIJ}
	W_{M,U}&=& \{w \in \cosets \mid  U \!= w.M \cap \fg_2 \}  \text{ and }\\
\nonumber	g_{M,U}(q) &=& \sum_{w \in W_{M,U} } q^{t_M(w)}.
\end{eqnarray}
Let $\cv_{M,U}$ be a graded vector space whose Poincar\'e polynomial is $g_{M,U}(q^2)$.  We consider $\cv_{M,U}$ as an $A(x)$-module with trivial action.

\begin{cor}\label{cor.P-orbits} \label{odd_vanishing}
	We have 
\begin{eqnarray}\label{eqn.U-decomp}
	H^*(\hess{x}{M}) \simeq   \displaystyle \bigoplus_{U \in \ideals^{gen}_2}  \cv_{M,U} \otimes H^*(X_U),
\end{eqnarray}
as $A(x)$-modules.  In particular, 	$H^i(\hess{x}{M}) =0$ for $i$ odd.
\end{cor}
\begin{proof}
As in \cite{dCLP1988}, $\hess{x}{M}$ admits an $\alpha$-partition of the nonempty $\hess{x, w}{M}$.
Since each $\hess{x, w}{M}$ has no odd Borel-Moore homology by Theorem \ref{no_odd_dclp} and Corollary \ref{relate_cohom}, 
the long exact sequence in Borel-Moore homology gives the isomorphism as vector spaces.  Then the naturality of the long exact sequence
and Corollary \ref{relate_cohom} yield the isomorphism as $A(x)$-modules.	 
Since $\hess{x}{M}$ and the $X_U$ are projective varieties, 
the Borel-Moore homology and singular cohomology coincide, and the result follows.
\end{proof}

\begin{Rem} \label{Property S}
	In fact, the varieties $\hess{x}{M}$ satisfy Property (S) from \S 1.7 in \cite{dCLP1988}.   Namely, switching to Borel-Moore  integral homology, we have 
	 $H_{i}(\hess{x}{M}) = 0$ for $i$ odd, $H_{i}(\hess{x}{M})$ has no torsion for $i$ even, and the cycle map from the $i$-th Chow group
	 of $\hess{x}{M}$  to $H_{2i}(\hess{x}{M})$ is an isomorphism.    These results follow from the fact that $X_U$ has Property (S)
	 for all $U \in  \ideals^{gen}_{2}$ as shown in \cite{dCLP1988} and Lemmas 1.8 and 1.9 in {\it loc.\! cit.}  
\end{Rem}

Let
\begin{equation} \label{modified_poincare}
\mathsf P(X) = \sum_j \dim \left( H^{2j}(X) \right) q^j
\end{equation}
denote the modified Poincar\'e polynomial of  a variety $X$ with no odd cohomology.
If a group $K$ acts on $X$ and $\chi  \in \Irr(K)$, let 
$$\mathsf{P}(X; \chi) = \sum_j \left(\chi\!:H^{2j}(X) \right) q^j,$$ 
where $(\chi\!:\!\chi') = \dim\Hom_K(\chi, \chi')$ denotes the multiplicity of $\chi$ in 
the $K$-representation $\chi'$.
Then Corollary \ref{eqn.U-decomp} immediately implies the following.

\begin{cor} \label{cor:poincare_decomp}
	We have 
	\[
	\mathsf{P}(\hess{x}{M}) = \sum_{U \in \ideals_2^{gen}} g_{M,U}(q) \mathsf{P}(X_{U}).
	\]
	and 
	\begin{eqnarray}\label{Frobenius_identity}
			\mathsf{P}(\hess{x}{M}; \chi) = \sum_{U \in \ideals_2^{gen}} g_{M,U}(q) \mathsf{P}(X_{U}; \chi).
	\end{eqnarray}
for all $\chi \in  \Ar$.
\end{cor}

\begin{Rem} \label{computation1}
	Corollary \ref{cor:poincare_decomp}  gives a way to compute $\mathsf{P}(\hess{x}{M})$ and $\mathsf{P}(\hess{x}{M}; \chi)$, which is feasible in small ranks.  There are three steps:  (1) determine the set $\ideals^{gen}_2$, (2) compute
	$\mathsf{P}(X_{U}; \chi)$ for each $U \in  \ideals^{gen}_2$, and (3) compute $g_{M,U}$.  
	
	For (1), methods from Fenn's thesis \cite{FennThesis} apply and Fenn and the second author wrote computer code to do this, which works up to rank 10.  For (3) the second author wrote code that is efficient up to rank 7 and can handle rank $8$ for the cases where $M \in \ideals$.    
	For (2), however, we only know how to do the computation in type $A$, small rank cases, and for certain nilpotent orbits in all cases.  The method for doing this comes from \cite{dCLP1988}, which we explain in \S\ref{sec.modular_law}.
		We have carried out this computation in several small rank cases 
	and give the example of $B_3$ for $M \in \ideals$ in \S\ref{section:example}. 
\end{Rem}


\section{Proof of Theorem \ref{ThmA}} \label{sec.proof_thmA}

We now restrict to the case where $M=I$ for $I \in \ideals$ and return to the map 
$\mu^I: G \times^B I \to \overline{\co_I}$ from 
the introduction.  Pushing forward the shifted constant sheaf yields
\begin{eqnarray}\label{eqn.gen_Springer.sheaf2}
	R\mu^I_*(\sh[N_I]) \simeq \bigoplus_{(\co, \cl)} IC(\bar{\co}, \cl)\otimes V^I_{\co, \cl}.
\end{eqnarray}
where the sum is over pairs $(\co,\cl) \in \allpairs$ 
consisting of a nilpotent orbit $\co \subset \overline{\co}_I$ and an irreducible local system $\cl$ on $\co$.
The $V^I_{\co, \cl}$ are graded complex vectors spaces.
In this section we will prove Theorem \ref{ThmA}:
if $V^I_{\bpairs} \neq 0$ in  \eqref{eqn.gen_Springer.sheaf2}, then $(\co, \cl)$ is an element of $\sppairs$,
the pairs that arise in the Springer correspondence~\eqref{eqn.Springer.correspondence}.

Taking stalks in  \eqref{eqn.gen_Springer.sheaf2} 
and using proper base change, the cohomology of the fibers of $\mu^I$ 
and local intersection cohomology are related by
\begin{equation}\label{proper_base1}
	H^{j+N_I}(\hess{x}{I}, \C)  = 
	\bigoplus_{\substack{(\co,\cl) \\ \co \subset \overline{\co}_I} }  
	\ch^{j}_x \left(IC(\bar{\co},\cl) \otimes V^I_{\co, \cl} \right)
\end{equation}
for $x \in \bar{\co_I}$.
We would like to write down the $A(x)$-equivariant version of this statement
as discussed in \cite{Borho-MacPherson1983} (see also \cite{Achar-book}).

For $x\in \cn$ and $\chi \in \Ar$, we sometimes write $(x,\chi)$ in place of $(\co_x, \cl_\chi)$.   
Given $(x, \chi) \in \allpairs$,  
we define the Laurent polynomial 
$m^I_{x, \chi}(q) \in \mathbb{N}[q,q^{-1}]$ 
so that the coefficient of $q^j$ is the dimension of the $j$-th graded component of $V^I_{x, \chi}$.
By the properties of perverse sheaves,  we have  $m^I_{x, \chi}(q^{-1}) = m^I_{x, \chi}(q)$. 
Let $\left(\cl_\chi: - \right)$ denote the multiplicity of $\cl_\chi$ in a local system on $\co_x$.
For $(x,\chi)$ and $(u, \phi)$ in $\allpairs$,
define 
$$c_{x,\chi}^{u, \phi}(q) =	\sum_{j \in \Z} \left(\cl_\chi: \ch^{j}IC(\bar{\co}_u,\cl_\phi)|_{\co_x} \right)q^j.$$

Then taking the dimension on both sides of the $A(x)$-equivariant version of \eqref{proper_base1} gives
\begin{equation}\label{poincare_decomp} 
	\mathsf P(\hess{x}{I}; \chi)(q^2)  =  q^{N_I} \sum_{(u,\phi)  \in \allpairs}  c_{x,\chi}^{u, \phi}(q) m^I_{u, \phi}(q). 
\end{equation}

Now suppose we know that $P(\hess{x}{I}; \chi)=0$ for some $\chi \in \Ar$.  Then since 
$c_{x,\chi}^{x, \chi} \neq 0$ and all the coefficients on the right side of \eqref{poincare_decomp} are nonnegative,
we conclude that $m^I_{x, \chi}=0$ and thus $V^I_{x, \chi} = 0$.  Hence, the proof of Theorem \ref{ThmA} 
is reduced to showing 
that  $	\mathsf P(\hess{x}{I}; \chi)=0$ for $(x,\chi) \not \in \sppairs$.

Next, by Corollary \ref{cor:poincare_decomp}, if we can show that 
$\mathsf P(X_U; \chi)=0$ for all $(x,\chi) \not \in \sppairs$ and all $U \in \ideals^{gen}_{2}$,
then it follows that $\mathsf P(\hess{x}{I}; \chi)=0$ for all $(x,\chi) \not \in \sppairs$.

For the case of $I =\fu$, which is the Springer resolution case, there is a stronger equivariant statement.
The Weyl group $W$ acts on $H^*(\cb_x)$ and this action commutes with the $A(x)$-action, so $W \times A(x)$ acts.
For $(x,\chi)$ and $(u, \phi)$ in $\sppairs$, 
\begin{equation}\label{poincare_decomp_W} 
	\mathsf P(\cb_x; V_{u, \phi} \otimes \chi)(q^2)  =  q^N c_{x,\chi}^{u, \phi}(q) 
\end{equation}
where $V_{u,\phi}$ is the irreducible $W$-representation corresponding to $(u,\phi)\in \sppairs$ 
in~\eqref{eqn.Springer.correspondence}.
Moreover, the left side (and the right side) will always be zero for those $(x, \chi) \not \in \sppairs$.
This was first established by Beynon--Spaltenstein by computer calculation in \cite{Beynon-Spaltenstein1984} for the exceptional groups
and by Shoji in \cite{Shoji1983} for the classical groups.  
Later, Lusztig gave a uniform framework that also included handling his Generalized Springer correspondence \cite{Lusztig1986}.
The resulting algorithm from \cite[\S 24]{Lusztig1986} is now known as the Lusztig--Shoji algorithm and it computes either side
of \eqref{poincare_decomp_W} knowing only the partial order on the orbits in $\cn$, the component groups $A(x)$, the Springer
correspondence, and the character table of $W$.

Returning to the proof of Theorem \ref{ThmA}, since the left side of \eqref{poincare_decomp_W} is zero for
$(x, \chi) \not \in \sppairs$, it follows that $P(\cb_x; \chi) = 0$  for $(x, \chi) \not \in \sppairs$. 
Therefore, if $(x,\chi) \not \in \sppairs$ and $U \in \ideals^{gen}_{2}$, then 
$\mathsf P(X_U; \chi)=0$ whenever $X_U$ appears in the decomposition of $H^*(\cb_x)$ in Corollary \ref{odd_vanishing}.
We can therefore finish the proof of the theorem if we can show that 
$X_U$ appears in the decomposition of $H^*(\cb_x)$ for all $U \in \ideals^{gen}_{2}$.  
In other words, we will show $g_{\fu, U} \neq 0$ for all $U \in \ideals^{gen}_{2}$, or equivalently, 
$W_{\fu, U} \neq 0$.  
This amounts to showing that, for all $U\in \ideals_2^{gen}$, there exists $w \in  \cosets$ such that $U = w.\fu \cap \fg_2$.  

By Lemma~\ref{lemma.coset-ideals} there is a map $\map: \cosets \to \ideals_2$ defined by
\begin{eqnarray} \label{c_map}
	\map(w) = w.\fu\cap \fg_2 \text{ for } w \in \cosets.
\end{eqnarray} 
In \cite[\S 3.7(b)]{dCLP1988} it is shown that $\map$ is surjective for the case when $\co_x$ is an even orbit, i.e., when $\fg_i=0$
for $i$ odd.  We now extend that result to show $\map$ is surjective for all nilpotent orbits.
\begin{lemma} \label{lem.g2-ideals} 
	Given $U \in \ideals_{2}$, there exists $w \in \cosets$ such that~$U= w.\fu \cap \fg_2$. 
\end{lemma}

\begin{proof} Let $U\in \ideals_2$. It suffices to show that there exists $w\in \cosets$ such that $\Phi_U = w(\Phi^+)\cap \Phi_2$.  
	When $\co_x$ is even, \cite[\S 3.7(b)]{dCLP1988} gives a construction of such a $w \in \cosets$.  
	Moreover,
	the $w$ constructed is the unique element achieving the largest possible value for $\ell(w)$ among those 
	$w \in \cosets$ satisfying $\Phi_U = w(\Phi^+)\cap \Phi_2$.  
	
	Suppose now that $\co_x$ is not even.  Consider the Lie subalgebra $$\fs:= \bigoplus_{i\in 2\Z} \fg_i.$$  
	Then $\fs$ is the centralizer in $\fg$ of an element of order $2$ in $G$, namely, the image of 
	$\left(\begin{smallmatrix}
		-1 & 0 \\
		0 & -1
	\end{smallmatrix}\right)$ under the map $SL_2(\C) \to G$ coming from the $\mathfrak{sl}_2$-triple defined in \S\ref{sec.preliminaries}.
	Hence $\fs$ is a reductive subalgebra and its simple roots are a subset of the extended simple roots of $\fg$ (see \cite{Carter1993}).  
	Let $\Phi_{\fs}$ be the root system of $\fs$ with positive roots $\Phi_{\fs}^+ =  \Phi^+ \cap \Phi_{\fs}$.
	Let $W_\fs$ denote the Weyl group of $\fs$ relative to $\ft \subset \fs$.
	
	The $\mathfrak{sl}_2$-triple for $x$ in $\fg$ lies in $\fs$ and so the $W_0$ defined relative to $\fg$ and $\fs$ coincide.  
	Since $U \subset \fg_2 \subset \fs$ and  $\co_x\cap \fs$ is an even orbit in $\fs$, where the lemma holds,
	there exists $w \in  \cosets \cap W_{\fs}$ such that $w(\Phi_{\fs}^+) \cap \Phi_2 = \Phi_U$.

	Now let $\sigma \in W$ be any element satisfying $\sigma^{-1}(\Phi^+_{\fs}) \subset \Phi^+$.   
	Then $\Phi^+_{\fs} \subset \sigma(\Phi^+) \cap \Phi_{\fs}$ and so $\sigma(\Phi^+) \cap \Phi_{\fs} = \Phi^+_{\fs}$.
	We claim that $$w\sigma(\Phi^+)  \cap \Phi_2 = \Phi_U.$$   
	Indeed, let $\beta \in \Phi_2$.
	Then $\beta \in w\sigma(\Phi^+)$ if and only if $w^{-1}(\beta) \in \sigma(\Phi^+)$.   	
	Since $w \in W_{\fs}$ and $\Phi_2\subset \Phi_{\fs}$, we have $w^{-1}(\Phi_2) \subset \Phi_{\fs}$. 
	Hence,  $w^{-1}(\beta) \in \sigma(\Phi^+)$ if and only if 
	$$w^{-1}(\beta) \in \sigma(\Phi^+) \cap \Phi_{\fs} = \Phi^+_{\fs}.$$
	This shows that $\beta\in w\sigma (\Phi^+)$ if and only if $\beta\in w(\Phi_{\fs}^+)$, 
	proving that $w\sigma(\Phi^+)\cap \Phi_2=w(\Phi_\fs^+)\cap \Phi_2$, as desired. 
	Finally, we may take $\sigma$ to be the identity of $W$ to see that $w$ itself satisfies $\Phi_U = w(\Phi^+) \cap \Phi_2$. 
\end{proof}	

The Lemma applies in particular to $U \in \ideals_{2}^{gen}$, completing the proof of Theorem \ref{ThmA}.  We can now write 
\eqref{eqn.gen_Springer.sheaf2} as
\begin{eqnarray}\label{eqn.gen_Springer.sheaf3}
	R\mu^I_*(\sh[N_I]) \simeq \bigoplus_{(\co, \cl) \in \sppairs} IC(\bar{\co}, \cl)\otimes V^I_{\co, \cl}.
\end{eqnarray}

\begin{Rem}
	The proof of Lemma~\ref{lem.g2-ideals} shows that for each $w \in  \cosets \cap W_{\fs}$ such that  $U= \map(w)$, 
	we have $\map(w\sigma)= \map(w)$ where $\sigma$ runs over the right coset representatives of $W_{\fs}$ in $W$ that lie in $\cosets$.   It follows from Corollary \ref{cor:poincare_decomp} for the Springer fiber case where $M = \fu$ that the index $[W:W_\fs]$ divides the Euler characteristic of $\cb_x$ when $\co_x$ is not an even orbit.  	
	
	For example, in  type $E_7$ there are $3$ involutions in $G$, up to conjugacy.  
	Curiously, among all non-even orbits, only the involution for which $\fs$ is of type $D_6 \times A_1$ occurs
	for the $\fs$ in the proof of  Lemma~\ref{lem.g2-ideals}.
	The Weyl group of the standard $D_6 \times A_1$ is exactly the stabilizer of the line through the highest root, showing   
	that $[W\!:\!W_\fs] = 63$, the number of positive roots in $E_7$.  We deduce that the Euler characteristic 
	of any Springer fiber of a non-even nilpotent element in $E_7$ is divisible by $63$,
	which was observed by Fenn.
\end{Rem}

It will be convenient sometimes to use a parametrization indexed by $\Irr(W)$ instead of $\sppairs$.
For each $\varphi \in \Irr(W)$, we can write $\varphi = \varphi_{\co,\cl}$ for a unique 
$(\co, \cl) \in \sppairs$ by the Springer correspondence \eqref{eqn.Springer.correspondence}.
We sometimes write $m^I_\varphi$ or $V^I_\varphi$ in place of
$m^I_{x, \chi}$ or $V^I_{\co_x, \cl_\chi}$, respectively.

Consider the partial order on $\Wr$ by 
$\varphi_{\co',\cl'} \preceq \varphi_{\co,\cl}$ if $\co' \subseteq \bar \co$.
We choose a linear ordering on $\Irr(W)$ that respects this partial order.
Define the matrix $\bf K$ to have entries coming from  
the left side of \eqref{poincare_decomp_W}:
$$K_{\varphi,\varphi'} =  \sum_j \left(\varphi \otimes \chi:  H^{2j}(\cb_x)\right) q^{j}$$
where $\varphi' = \varphi_{\co_x,\cl_\chi}$.  The entries of $\bf K$ lie in $\mathbb N[q]$ and
$\bf K$ is lower triangular, with powers of $q$ along the diagonal. The matrix
$\bf K$ is computable by the Lusztig--Shoji algorithm \cite[\S 24]{Lusztig1986}.
The entries of $\bf K$ are sometimes referred to as the Green functions of $G$.  In type $A$,
they coincide with modified Kostka--Foulkes polynomials~\cite{Garsia-Procesi1992}.

Set $c_I = \dim \fu - \dim I$, the codimension of $I$ in $\fu$.  
Then since $c_I = N - N_I$, we can rewrite Equation~\eqref{poincare_decomp} as
\begin{equation}\label{poincare_decomp2} 
	q^{2c_I}\,\mathsf P(\hess{x}{I}; \chi)(q^2)  =   \sum_{\varphi \in \Irr(W)} q^{c_I}  m^I_\varphi(q)  {\bf K}_{\varphi, \varphi'}(q^2).
\end{equation}
where $\varphi' = \varphi_{\co_x,\cl_\chi}$.
Since the other expressions in \eqref{poincare_decomp2} involve even exponents, 
it follows that the exponents of $q^{c_I}  m^I_\varphi(q)$ are also even, allowing us to make the following definition.

\begin{definition}\label{key_polys}
For $\varphi \in \Irr(W)$ and $I \in \ideals$, define $f^I_\varphi(q)$ by $f^I_{\varphi}(q^2) = q^{c_I } m^I_{\varphi}(q)$.
\end{definition}

\begin{prop}\label{prop:key_polys}
	We have $f^I_{\varphi} \in  \mathbb N[q]$ and its coefficients are symmetric about $q^{c_I/2}$.
\end{prop}
\begin{proof}
	The highest power of $q$ on the left in \eqref{poincare_decomp2} is $2c_I + 2\dim \hess{x}{I}$, 
	each term in the sum is bounded by this value.  In particular this holds for $\varphi' = \varphi$.
	In that case, ${\bf K}_{\varphi', \varphi'}(q^2) = q^{2\dim (\cb_x)}$ and so
	$$2 \deg(f^I_{\varphi'}) + 2\dim (\cb_x) \leq 2c_I + 2\dim \hess{x}{I}$$
    Since $\hess{x}{I} \subset \cb_x$ and so $\dim \hess{x}{I} \leq \dim \cb_x$,
    we get $\deg(f^I_{\varphi'}) \leq c_I$.   Hence, 
    the exponents of $m^I_{\varphi'}$ are bounded by $c_I$ 
    and therefore below by  $-c_I$ by the $q \to q^{-1}$ symmetry of $m^I_{\varphi'}$.
	Hence, $f^I_{\varphi'} \in  \mathbb N[q]$ and its coefficients are symmetric about $q^{c_I/2}$.	
\end{proof}

The polynomials $f^I_\varphi$, or their transformations by a fixed matrix independent of $I$, 
make several appearances in the rest of this paper.

The results in this section where obtained in \cite{Borho-MacPherson1983}
in the parabolic setting, i.e., for $I = \fu_P$ where $P$ is any parabolic subgroup of $G$. 
Borho and MacPherson also did more, 
giving a formula for $f^{\fu_P}_\varphi$.  Namely,
\begin{equation}\label{parabolic_case}
	f^{\fu_P}_\varphi = \mathsf{P}(P/B) (\varphi: \Ind_{W_P}^W(\mbox{sgn})).  
\end{equation}
where $\mbox{sgn}\in \Irr(W)$ denotes the sign representation of $W$ and $W_P$ denotes the Weyl group of $P$.

\begin{Rem}
	Although less efficient than the Lusztig--Shoji algorithm, our methods give an inductive way to compute $\bf K$
	while also computing $m^I_\varphi$, or equivalently $f^I_\varphi$.  
	The induction starts at the zero orbit and moves up in the partial order on $\sppairs$.  
	The procedure relies on the dimension constraints for the IC-sheaves and the symmetry of the coefficients of $m^I_\varphi$.
	We need to know $H^*(\hess{x}{I})$ for all $x$ and enough $I$ and then we can use \eqref{poincare_decomp2}.  	
	This inductive process is analogous to the usual method for computing 
	the Kazhdan--Lusztig polynomials, i.e., the IC stalks of Schubert varieties, using the Bott--Samelson resolutions.
\end{Rem}

\section{Generalized Grothendieck--Springer setting}\label{sec.gen_grothendieck}

In this section we prove a conjecture of Brosnan and 
show that the polynomials $f^I_\varphi$ from Definition \ref{key_polys} control the decomposition of 
the pushforward of the shifted constant sheaf when 
we move from the setting of  $I \in \ideals$ to that of $H \in \mathcal{H}$.

\subsection{Brosnan's conjecture}
Let $H\in \mathcal{H}$ and consider the map 
\begin{eqnarray}\label{eqn.muH}
	\mu^H: G\times^B H \to \fg, (g,x) \mapsto g. x.
\end{eqnarray}
This map is proper and since $\fb\subseteq H$, the image of $\mu_H$ is $\fg$.  
Let $N_H = \dim G/B + \dim H$, the dimension of the smooth variety $G\times^B H$. 

When  $H=\fb$, this map is the Grothendieck--Springer resolution, which Lusztig \cite{Lusztig1981} showed was a small map.  In particular, $N_b = \dim \fg$.
Since the map is small, $R\mu^\fb_* (\sh[N_\fb])$ decomposes into a sum of irreducible perverse sheaves on $\fg$ with maximal support.  More precisely
\begin{eqnarray}\label{eqn.decomp1}
	R\mu^\fb_* (\sh[N_\fb]) = \bigoplus_{\varphi \in \Irr(W)} IC(\fg , \cm_\varphi )\otimes \varphi
\end{eqnarray}
where $\cm_\varphi$ is the irreducible local system supported on the regular semisimple elements $\fg_{rs}$ of $\fg$ 
corresponding to $\varphi \in \Irr(W)$ \cite[Lemma 8.2.5]{Achar-book}.

We wish to generalize Equation \eqref{eqn.decomp1} to any $H \in \mathcal{H}$ and to compute the sheaf $R\mu^H_* (\sh_{H}[N_H])$, as we did for the case of $I \in \ideals$. 
Recall in that situation, as a consequence of Theorem \ref{ThmA} that \eqref{eqn.gen_Springer.sheaf2}
becomes
\begin{eqnarray}\label{decomp_springerized}
	R\mu^I_*(\sh[N_I]) \simeq \bigoplus_{(\co, \cl) \in \sppairs} IC(\bar{\co}, \cl)\otimes V^I_{\co, \cl},
\end{eqnarray}
where $V^I_{\co, \cl}$ is a $\Z$-graded vector space. 
For $\varphi \in \Irr(W)$,
we will write  $V^I_\varphi$ for  $V^I_{\co, \cl}$ where $\varphi = \varphi_{\co, \cl}$ under the Springer correspondence.

Our result is the following theorem, originally conjectured by Brosnan (see \cite[Conjecture 5.2.2]{Xue2020} and  \cite{Vilonen-Xue2021}). 
\begin{thm} \label{thm.sheaves.ss} Let $H\in \mathcal{H}$. 
	Let $I =H^\perp$, the annihilator of $H$ under the Killing form. 
	There is an isomorphism 
	\begin{eqnarray} \label{H_decomp}
	R\mu^H_* (\sh_{H}[N_H]) \simeq \bigoplus_{\varphi \in {\mbox{Irr}}(W)} IC(\fg, \cm_{\varphi}) \otimes V_{\varphi \otimes \mbox{sgn}}^I
	\end{eqnarray}
	in the derived category of $G$-equivariant perverse sheaves on $\fg$. 
	 In particular, every simple summand of $R\mu^H_* (\sh_{H}[N_H])$ is a simple perverse sheaf on $\fg$ with full support.
\end{thm}

\begin{proof}  To prove the result, we apply the Fourier transform (see \cite[Cor.~6.9.14]{Achar-book}) to obtain 
	\begin{eqnarray}\label{eqn.Fourier}
		\mathfrak{F}(R\mu^I_* \sh_I[N_I]) =  R\mu^H_* (\sh_{H}[N_H]),
	\end{eqnarray}
	where $\mathfrak{F}$ maps each simple summand of 
	$R\mu^I_* (\sh_I[N_I])$  in \eqref{decomp_springerized}
	to a simple summand of $R\mu^H_* (\sh_{H}[N_H])$.
	Next, for $(\co, \cl) \in \sppairs$, we have
	\[
	\mathfrak{F}(IC(\bar \co, \cl)) = IC(\fg, \cm_{\varphi \otimes \mbox{sgn}}) \textup{ where }  \varphi = \varphi_{\co, \cl}
	\]
	(see Sections 8.2, 8.3, and specifically equation (8.3.2) of \cite{Achar-book}).  The result follows.
\end{proof}
Theorem \ref{thm.sheaves.ss} generalizes \cite[Theorem 3.6]{Balibanu-Crooks2020} to all Lie types.

\subsection{Monodromy action} \label{monodromy}
Let $s \in \ft$ be a regular semisimple element.   The variety $\hess{s}{H}$ defined as in~\eqref{definition.varieties} is called a regular semisimple Hessenberg variety.
There is an action of $W$ on $H^*(\hess{s}{H})$ arising from monodromy \cite{Brosnan-Chow2018}, \cite{Balibanu-Crooks2020}.  
Then taking stalks in \eqref{H_decomp} at $s$ and using $W$-equivariant proper base change gives
\begin{equation}
H^{*+N_H-\dim \fg}(\hess{s}{H}) \simeq  \bigoplus_{\varphi \in \Irr(W)} \varphi \otimes V_{\varphi \otimes \mbox{sgn}}^{H^\perp}
\end{equation}
as $W$-modules.  Here, $\varphi$ is in degree $0$ on the right and carries the $W$-action.  We have used that
$\ch^j_sIC(\fg, \cm_{\varphi})$  is zero, except for $j = -\dim \fg$, where it equals $\varphi$.

Now $N_H -\dim \fg = \dim H - \dim \fb$, which is the dimension of $\hess{s}{H}$. 
We deduce that for $\varphi \in \Irr(W)$ that
\begin{equation}
\mathsf P(\hess{s}{H}; \varphi)(q^2) = q^{\dim H - \dim \fb}  m^{H^\perp}_{\varphi \otimes \mbox{sgn}}(q).
\end{equation}
Since  $\dim H - \dim \fb = \dim \fu - \dim H^{\perp} = c_{H^{\perp}}$, we immediately have
\begin{prop} \label{ss_hessy}
	For $H \in \mathcal H$ and $\varphi \in \Irr(W)$, we have $\mathsf P(\hess{s}{H}; \varphi) = f^{H^\perp}_{\varphi \otimes \mbox{sgn}}.$
\end{prop}
Note that Proposition \ref{ss_hessy} gives another proof that $f^{I}_{\varphi}$ is a polynomial,
while Proposition \ref{prop:key_polys} gives a (new) proof that 
$\mathsf P(\hess{s}{H}; \varphi)$ is palindromic.

\subsection{Nilpotent Hessenberg varieties}

We can also take the stalks in  \eqref{H_decomp}  at elements of $\cn$ and get an interesting result.
First, we need the fact that 
$$IC(\fg, \cm_{\varphi})|_\cn [-\dim \ft]\simeq IC(\bar \co, \cl)$$  
where $\varphi = \varphi_{\co, \cl} $ (for example, see \cite[Lemma 8.3.5]{Achar-book}).  
Next, $N_H - N = \dim H -\dim \fu$. But with the extra shift by $\dim \ft$ from above,  
we get the analogue of \eqref{poincare_decomp2}.

\begin{prop} \label{prop:nilpo_hessy}
 For $x \in \cn$ and $\chi \in \Ar$, we have
\begin{equation}\label{nilp_hessy_decomp}
	\mathsf P(\hess{x}{H}; \chi) =  \sum_{\varphi \in \Irr(W)}   f^{H^\perp}_{\varphi \otimes \mbox{sgn}} {\bf K}_{\varphi, \varphi'}
\end{equation}
where $\varphi' = \varphi_{\co_x,\cl_\chi}$.
\end{prop}

Proposition \ref{prop:nilpo_hessy} was known in the parabolic case.  In  \cite{Borho-MacPherson1983}, 
Borho and MacPherson consider the restriction $\mu^H_\cn$ of $\mu^H$ to 
$X_\cn = G \times^B (H \cap \cn)$.
Since the intersection $\fp \cap \cn$ is rationally smooth for $H = \fp$, 
they obtained the stronger statement 
\begin{equation} \label{H_decomp2}
	R(\mu_\cn^\fp)_*(\sh[\dim X_\cn])  \simeq 
	\bigoplus_{(\co, \cl) \in \sppairs} IC(\bar{\co}, \cl)\otimes V^{\fu_P}_{\varphi_{\co, \cl} \otimes \small \mbox{sgn}}.
\end{equation}
We suspect that \eqref{H_decomp2} holds for all $H \in \mathcal H$, 
presumably because $X_\cn$ is rationally smooth for all $H \in \mathcal{H}.$
In type $A$, Proposition \ref{prop:nilpo_hessy}, or perhaps the stronger version in 
\eqref{H_decomp2}, is an unpublished theorem of Tymoczko and MacPherson \cite[pg.~2882]{Oberwolfach}.

Finally,  we note that when $x$ is regular nilpotent, 
the only term in \eqref{nilp_hessy_decomp} for which  ${\bf K}_{\varphi, \varphi'}$
is nonzero occurs 
when $\varphi$ is the trivial representation, in which case the value is $1$.
Hence $f^{H^\perp}_{\mbox{sgn}} = 	\mathsf P(\hess{x}{H})$.
When $x$ is regular nilpotent, $\mathsf P(\hess{x}{H})$ has a formula as a product of $q$-numbers
depending only on the roots $\Phi_H \cap \Phi^-$ \cite{AHMMS2020}, \cite{Sommers-Tymoczko2006}.

\subsection{The dot action and LLT representations} 
In \cite{Tymoczko2008}, Tymoczko defined a Weyl group representation on the ordinary cohomology of regular semisimple Hessenberg varieties, called the dot action.
The representation of $W$  on $H^*(\hess{s}{H})$
from \S \ref{monodromy} was shown to coincide with Tymoczko's dot action representation by Brosnan and Chow in~\cite{Brosnan-Chow2018} in type A, and their proof was adapted to all Lie types by B\u{a}libanu and Crooks in~\cite{Balibanu-Crooks2020}.

The dot action  is closely related to another $W$-representation via a tensor product formula,
which first arose in Procesi's study of the toric variety associated to the Weyl chambers from~\cite{Procesi1990} and 
was later treated by Guay--Paquet in type $A$  \cite{Guay-Paquet2016}.   
We summarize the key properties we need here and include the details in Appendix~\ref{appendix}.

Let $\cc$ denote the coinvariant algebra of $W$, with the reflection representation in degree $1$.  
Then $\cc \simeq H^{*}(G/B)$ as graded $W$-representations, up to the doubling of degrees.
If $g(q) = \sum a_i q^i$ is a polynomial, we regard it as a graded representation consisting 
of $a_i$ copies of the trivial representation in
degree $a_i$. 
The following result is a direct generalization of \cite[Theorem 2]{Procesi1990} and \cite[Lemma 168]{Guay-Paquet2016}.

\begin{prop}\label{prop.LLT}
	For each Hessenberg space $H\in \ch$ there exists a unique graded $W$-representation $\LLT_H$
	 satisfying
	\begin{eqnarray}\label{eqn.LLT-dot-formula}
		\mathsf{P}(G/B) \otimes \LLT_H \simeq \cc \otimes H^*(\hess{s}{H}).
	\end{eqnarray}
Furthermore, $\LLT_H$ is nonzero only in nonnegative degrees.
\end{prop}

We call $\LLT_H$ the LLT representation; the reason for this terminology is that, in Type A, 
the Frobenius characteristic of $\LLT_H$ is a unicellular LLT polynomial (see Corollary~\ref{cor.typeA-characters} below).  

The coinvariant algebra $\cc$ carries the regular representation and the tensor product of any representation  of dimension $d$
with the regular representation is the direct sum of $d$ copies of the regular representation.  Therefore when $q=1$ both sides of~\eqref{eqn.LLT-dot-formula} are isomorphic to $|W|$ copies of the regular representation.   Thus $\LLT_H$ is a graded version of the regular representation.
At the same time, forgetting the $W$-actions, 
$H^*(\hess{s}{H})$ and $\LLT_H$ coincide as graded vector spaces 
since $\mathsf P(G/B)$ measures the dimension of the components of $\cc$.

To compute $\LLT_H$ from $H^*(\hess{s}{H})$ requires knowing the matrix ${\bf \tilde \Omega}$
with entries 
$${\bf \tilde \Omega}_{\varphi, \varphi'} = \sum_j \left( \varphi \otimes \varphi': \cc^j \right)q^j.$$
The matrix ${\bf \tilde \Omega}$ is closely related to the matrix needed as input to the Lusztig--Shoji algorithm.
Define polynomials $g^H_\varphi(q)$ by 
$$g^H_\varphi(q) = \sum_j (\varphi: LLT_H)q^{2j}.$$
Then 
\begin{equation} \label{compute_llt}
	{\mathsf{P}(G/B)} g^H_{\varphi} = \sum_{\tiny{\varphi' \in \Wr}} P(\hess{s}{H}; \varphi'){\bf \tilde \Omega}_{\varphi, \varphi'}  
	  =     \sum_{\tiny \varphi' \in \Wr} f^{H^\perp}_{\varphi'} {\bf \tilde \Omega}_{\varphi, \varphi'} .
\end{equation}
We have used Proposition~\ref{ss_hessy} and the fact that 
$$\Hom_W(\varphi \otimes \varphi', \cc^j) \simeq \Hom_W(\varphi, \varphi' \otimes \cc^j)$$ since 
representations $\varphi~\in~\Wr$ satisfy $\varphi^*\simeq \varphi$.

For \S \ref{sec.modular.revisited}  
we need to know the  $H^*(\hess{s}{H})$ and $LLT_H$ in the parabolic case, i.e., when $H = \fp$ the Lie algebra 
of the parabolic subgroup $P$.   
Let $\cc_P$ be the coinvariant algebra of $W_P$ as a Coxeter group, a graded representation of $W_P$.
\begin{prop} 	\label{cor.LLT-parabolic}
	 In the parabolic setting, we have
	\begin{enumerate}
		\item $H^*(\hess{s}{\fp})  \simeq  \mathsf P(P/B) \otimes {\Ind}_{W_P}^W( \bf 1)$. 
		\item $LLT_{\fp} \simeq \mathrm{Ind}_{W_P}^W(\cc_P)$. 
	\end{enumerate}
	 where the modules on the left are zero for odd degrees and the 
	 component in degree $2i$ on the left matches the one in degree $i$ on the right.
	\end{prop}
\begin{proof} 
	Part (1) follows directly from Equation \eqref{parabolic_case} and Proposition~\ref{ss_hessy}.
	
	For (2), we have the isomorphism $\cc_P \simeq H^*(P/B)$, as $W_P$-representations, up to doubling of degrees.  
	Then the fibre bundle of $G/B$ over $G/P$ with fiber $P/B$ gives rise to the isomorphism
	$$H^*(G/B) \simeq H^*(G/P) \otimes H^*(P/B)$$ 
	as $W_P$-representations where the action on $H^*(G/P)$ is trivial.
	Hence, $\cc \simeq \mathsf P(G/P) \cc_P$ as $W_P$-representations.
	Inducing up to $W$, we have 
	$$\Ind_{W_P}^W (\cc)  \simeq \mathsf{P}(G/P) \otimes \Ind_{W_P}^W(\cc_P ) $$
	as $W$-modules, or equivalently,
	$$\cc \otimes \Ind_{W_P}^W ({\bf 1})  \simeq \mathsf{P}(G/P) \otimes \Ind_{W_P}^W(\cc_P ) $$
	since $\cc$ is a $W$-representation.
	Now using Part (1) and the fact that $\mathsf P(G/B) = 	\mathsf{P}(G/P) \mathsf{P}(P/B)$,
	the result follows.
\end{proof}


\section{The modular law}\label{sec.modular_law}

Let $M\subseteq \fg$ be a $B$-invariant subspace. 
As mentioned in Remark \ref{computation1}, 
the results of \S\ref{sec.decomposition} give a method to compute the isotypic component 
$H^*(\hess{x}{M})^\chi$ if we can compute the isotypic component $H^*(X_U)^\chi$ for all $U \in \ideals^{gen}_{2}$, where $\ideals^{gen}_{2}$ is defined as in \S\ref{building_blocks} using the grading induced by $x$. 
  
In \cite{dCLP1988}, a method is given to compute $H^*(X_U)^\chi$ that works for all distinguished nilpotent elements 
in the exceptional groups. Although it does not work in general, it is a powerful technique.  
In this section we prove a generalization of this method:  
for certain triples of $I_2 \subset I_1 \subset I_0$ of ideals in $\ideals$,
knowing any two of the $H^*(\hess{x}{I_i})^\chi$ determines the third.  We call this relation the geometric modular law.

We were led to this generalization after seeing the type $A$ combinatorial version 
in \cite{Abreu-Nigro2020} and \cite{Guay-Paquet2013},
where the relation is known as the modular law.  
In Proposition \ref{prop.comb-to-geometric} below, we show that our geometric modular law implies the combinatorial one.

\subsection{The basic move}\label{basic_move_defn}
We first define a relation on ideals $I_1, I_0 \in \ideals$ as in \cite{Fenn-Sommers2020}.

\begin{definition}\label{def.A2triple} 
     Two ideals $I_1, I_0 \in \ideals$ are related by the {\bf basic move} if $I_0 = I_1 \oplus \fg_{\beta}$ for $\beta \in \Phi^+$ 
     and there exists $\alpha\in \Delta$ such that 
	\begin{enumerate}
		\item $\left< \beta, \alpha^\vee \right> = -1$, and
		\item  The set $\Phi_{I_0}$ is invariant under the simple reflection $s_\alpha \in W$.
	\end{enumerate}
\end{definition}
The second condition is equivalent to $I_0$ being $P_\alpha$-stable, where $P_\alpha$ is 
the parabolic subgroup containing $B$ corresponding to $\alpha$.

If ideals $I_1\subset I_0$ satisfy $I_0 = I_1 \oplus \fg_{\beta}$, then 
$\beta$ is a minimal root in $\Phi_{I_0}$ under the partial order on positive roots.
If $I_1$ and $I_0$ are related by the basic move, then condition (1) implies that 
$s_\alpha(\beta) = \alpha+\beta$, which is a positive root bigger than $\beta$
in the partial order; hence, $\alpha+\beta \in \Phi_{I_1}$.  In fact, $\alpha+\beta$ is a minimal root of $\Phi_{I_1}$.  
Suppose otherwise; then $\Phi_{I_1}$ contains a positive root $\gamma = \alpha + \beta - \alpha'$ 
for a simple root $\alpha'$.  Now, $\alpha \neq \alpha'$ since $\beta \not\in \Phi_{I_1}$.   
Since $\left< \alpha', \alpha^\vee \right> \leq 0$ for any distinct simple roots
and $\left< \alpha + \beta, \alpha^\vee \right> = 1$  by condition (1), then $\left< \gamma, \alpha^\vee \right> \geq 1$.
But then $s_\alpha(\gamma)  < \beta$.  Since $\gamma \in \Phi_{I_1}  \subset \Phi_{I_0}$ and 
$$s_\alpha(\Phi_{I_0})= \Phi_{I_0}$$ by condition (2), this means $s_\alpha(\gamma) \in \Phi_{I_0}$
and we obtain a contradiction to $\beta$ being a minimal root in $\Phi_{I_0}$.

Thus we can define $I_2 \in \ideals$ to be the subspace
satisfying $I_1=I_2 \oplus \fg_{\alpha+\beta}$.  
It is clear that $I_2$ is $P_{\alpha}$-stable.  
We call $I_2\subset I_1\subset I_0$ a {\bf modular triple}, or just a {\bf triple}.
These triples were first constructed in \cite[\S 2.7]{dCLP1988}.

Since $P_{\alpha}.I_1 = I_0$, we have $\co_{I_1} = \co_{I_0}$,
while the orbit $\co_{I_2}$ need only satisfy $\co_{I_2} \subset \overline \co_{I_0}$.

\begin{example}\label{ex.typeAmove} 
	Consider the $A_3$ root system with $\Delta = \{\alpha_1, \alpha_2, \alpha_3\}$, where $\alpha_1$ and $\alpha_3$ orthogonal. 
	Let $I_0 \in \ideals$ be such that $\Phi_{I_0}$ has one minimal root $\beta= \alpha_2$.  
	Then $\beta$ satisfies condition (1) with respect to either $\alpha_1$ or $\alpha_3$.  
	Hence $I_1$ and $I_0$ are related by the basic move where $\Phi_{I_1}= \{ \alpha_1+\alpha_2, \alpha_2+\alpha_3, \alpha_1+\alpha_2+\alpha_3 \}$.  But there are two different triples
	that arise: $I_2$ will satisfy $I_1 = I_2 \oplus \fg_{\alpha_1+\alpha_2}$   or 
	$I_1 = I_2 \oplus \fg_{\alpha_2+\alpha_3}$ 
	depending on whether $\alpha = \alpha_1$ or $\alpha = \alpha_3$, respectively.
\end{example}

\subsection{Geometric modular law for the cohomology of fibers}
Whenever three subspaces form a modular triple,  they
satisfy the three conditions in \S 2.7 of \cite{dCLP1988} with 
$U'' \subset  U \subset U'$ in place of $I_2 \subset I_1 \subset I_0$ and $M = G$, $H = B$, and $P=P_{\alpha}$.

Fix $x \in \cn$ and let $X_i = \hess{x}{I_i}$.  
When $x \in \co_{I_0}$, it is shown in \cite[Lemmas 2.2]{dCLP1988}
that the $X_i$ are smooth ($X_2$ can be empty), as was noted in \S\ref{building_blocks}.  
In \cite[Lemma 2.11]{dCLP1988} it is proved that
there is a geometric relationship among the three varieties, which we now show holds
for all  $x \in \cn$, not just for $x \in \co_{I_0}$, 
and from this we deduce Proposition \ref{geom_modular_fiber}.
Since the $X_i$ are no longer smooth in general, 
our argument relies on the fact that 
$\hess{x}{I}$ has no odd homology for any $x \in \cn$ by Corollary \ref{odd_vanishing}.  We are now ready to prove the geometric modular law, which we restate here for the reader's convenience.

\begin{customprop}{1.2}[The geometric modular law]
	Let $I_2\subset I_1\subset I_0$ be a modular triple and let $x\in \cn$.   
	Write $X_i$ for $\hess{x}{I_i}$.
	Then
	\begin{eqnarray} \label{geometric_modular1}
		H^j(X_1) \oplus H^{j-2}(X_1) \simeq H^j(X_0) \oplus H^{j-2}(X_2) \,\text{ for all }\, j \in \mathbb Z.
	\end{eqnarray}
\end{customprop}

\begin{proof}  
	We again use Borel-Moore homology until the final step. 
	By Corollary \ref{odd_vanishing}, the odd homology of all $X_i$ vanish, so we need only consider $j$ even.
	Let $\alpha \in \Delta$ from the definition of the modular triple.  
	Consider 
	\[
	Z = \{(gB, g'B) \in G/B\times G/B \mid g^{-1}. x \in I_1 \textup{ and } g^{-1}g'\in P_\alpha \},
	\]
	as in Lemma 2.11 in  \cite{dCLP1988}.  Then $Z$ is a $\mathbb P^1$-bundle over $X_1$ by forgetting the second factor.
	Since $X_1$ has no odd homology, and neither does $\mathbb P^1$, the Leray spectral sequence degenerates and yields
	\begin{eqnarray} \label{p1_bundle_homology}
		H_j(Z) \simeq  H_j(X_1) \oplus H_{j-2}(X_1) \text{ for all  } j \in \mathbb N.
	\end{eqnarray}
	
	Next, as in {\it loc cit}, the variety $Z$ maps to a variety $Z'$, which is $\mathbb P^1$-bundle over $X_0$, 
	by sending $(gB, g'B) \in Z$ to $(gB, g'B)  \in Z'$, where 
	\[
	Z':=\{(gB, g'B) \in G/B\times G/B \mid g'^{-1}.x \in I_0 \textup{ and } g^{-1}g'\in P_\alpha \}.
	\]
	This works since $g' = gp$ for some $p \in P_\alpha$, so 
	$g'^{-1}.x = p^{-1}g^{-1} \in I_0$ if $g^{-1}.x \in I_1$ because $P_\alpha.I_1 = I_0$.
	Inside $Z'$ consider the subvariety $Y$ consisting of those
	$(gB, g'B) \in Z'$ where the line $g I_1/g'I_2 \subset g'I_0/g'I_2$ contains $x$.
	Then $Z$ is isomorphic to $Y$.  
	
	At the same time, $Y$ maps surjectively to $X_0$ and the pre-image of $X_2$
	is a  $\mathbb P^1$-bundle $E$ over $X_2$.  The complement of $E$ in $Y$ is isomorphic to $X_0 \backslash X_2$.
	In \cite{dCLP1988}, the proof ends since $X_2$ and $X_0$ are smooth and one knows the singular cohomology of $Y$, which is a blow-up of $X_0$ over $X_2$. 
	To get around the lack of smoothness in the general case, we again use that the $X_i$ have no odd Borel-Moore homology.
	
	First, $H_j(E) \simeq  H_j(X_2) \oplus H_{j-2}(X_2)$ for all $j$ as in \eqref{p1_bundle_homology} since $X_2$ has no odd homology. 
	Next, when $j$ is even, there are two long exact sequences in Borel-Moore homology.  For $X_2 \subset X_0$, we have 
	$$ 0 \to H_{j+1}(X_0 \backslash X_2)  \to  H_{j}(X_2)  \to  H_{j}(X_0)  \to H_{j}(X_0 \backslash X_2) \to 0$$ 
	and for $E \subset Y$, we have
	$$ 0 \to H_{j+1}(Y \backslash E)  \to  H_{j}(E)  \to  H_{j}(Y)  \to H_{j}(Y \backslash E) \to 0$$
	since $E$ has no odd homology. 
	Now, $Y \backslash E \simeq X_0 \backslash X_2$ and so it follows that 
	$$H_j(Y) \oplus H_j(X_2) \simeq H_j(E) \oplus H_j(X_0)$$ 
	as $\Q$-vector spaces.
	Thus $H_j(Y) \simeq H_j(X_0) \oplus H_{j-2}(X_2)$.
	The result follows from  \eqref{p1_bundle_homology} and the isomorphism $Z \simeq Y$.
	We can switch back to singular cohomology since the $X_i$ are projective varieties. 
\end{proof}

\subsection{Implications of the modular law}
It follows from Proposition \ref{geom_modular_fiber} that if $I_2\subset I_1\subset I_0$ is a triple of ideals, then 
\begin{equation}\label{first_kind}
	(1+q)\mathsf{P}(\hess{x}{I_1};\chi)  =  \mathsf{P}(\hess{x}{I_0};\chi) +q \mathsf{P}(\hess{x}{I_2};\chi). 
\end{equation}
This leads us to define
\begin{definition} \label{def:modular_law_general}
A collection of objects $\{F^I\}_{I \in \ideals}$ each carrying an action of $\Q(q)$
is said to {\bf satisfy the modular law} if 
\begin{equation}\label{second_kind}
	(1+q)F^{I_1}  =  { F}^{I_2} + q{ F}^{I_0}
\end{equation}
whenever $I_2\subset I_1\subset I_0$ is a triple of ideals.
\end{definition}
The $F^I$ could be polynomials or Laurent polynomials in $\Q(q)$ or a vector of such polynomials indexed by
$\Wr$ or, equivalently, $\sppairs$.  We could also take $F^I$ to be a graded representation of $\Wr$, which in type $A$
amounts to a symmetric function with coefficients in $\Q(q)$.

Notice that we have reversed the role of $I_0$ and $I_2$ in \eqref{second_kind} as compared to \eqref{first_kind}.

\begin{Rem}\label{linear_in_modular}
If $\{F^I_1\}$ and $\{F^I_2\}$ satisfy the modular law, so 
does $\{aF^I_1+ bF^I_2\}$ for any $a,b \in \Q(q)$.
\end{Rem}

\begin{prop} \label{mod_laws} 
	The following polynomials satisfy the modular law in \eqref{second_kind}.
	\begin{enumerate}
		\item $\mathsf{P}(\hess{x}{I};\chi)(q^{-1})$ for $(x, \chi) \in \sppairs$. 
		\item $f_\varphi^I(q)$ for $\varphi \in \Wr$.
		\item $g_\varphi^{I^\perp}(q)$ for $\varphi \in \Wr$.
		\item $\mathsf{P}(\hess{x}{I^{\perp}};\chi)(q)$ for $(x, \chi) \in \sppairs$. 
	\end{enumerate}
\end{prop}
\begin{proof}
Statement (1) follows \eqref{first_kind} by substituting $q^{-1}$ for $q$ and then multiplying by $q$.

For (2), the determinant of $\bf K$ equals $q^m$ for some $m \in \mathbb N$.  So $\bf K$ is invertible (and
in fact the entries of $q^m{\bf K}^{-1}$ are polynomials).  We want to convert  \eqref{poincare_decomp2} into a matrix equation.
To that end, we construct vectors $(\mathsf P(\hess{x}{I}; \chi))$ and $\left(q^{-c_I}f^{I}_\varphi\right)$
using the linear order on $\sppairs$ and $\Wr$, respectively, from \S\ref{sec.proof_thmA}.  
Then   \eqref{poincare_decomp2} becomes the matrix-vector equation
$$(\mathsf P(\hess{x}{I}; \chi)) =   (q^{-c_I}f^{I}_\varphi) {\bf K}$$
and therefore
$$\left(q^{-c_I}f^{I}_\varphi\right) =(\mathsf{P}(\hess{x}{I};\chi)) {\bf K^{-1}}.$$
Now part (1) and the Remark \ref{linear_in_modular} 
imply that $q^{-c_I}f^{I}_\varphi(q^{-1})$ satisfies the modular law for all $\varphi \in \Wr$.
Multiplying by $q^{c_{I_2}}$ we have  
$$(1+q) \cdot q^{-1}f^{I_1}_\varphi(q^{-1})) =  f^{I_2}_\varphi(q^{-1}) + q \cdot q^{-2}f^{I_0}_\varphi(q^{-1})$$ 
and the result follows by replacing $q^{-1}$ with $q$.

Statement (3) now follows from (2) and \eqref{compute_llt} and Remark \ref{linear_in_modular}.  
Similarly (4) follows from Proposition~\ref{prop:nilpo_hessy}. 
\end{proof}

\begin{Rem}
It is also possible to prove that $\mathsf{P}(\hess{x}{H};\chi)$ satisfies the modular law by adapting the geometric proof
in Proposition \ref{geom_modular_fiber} to the varieties $\hess{x}{H}$.
\end{Rem}

\subsection{Alternative formulation of Proposition \ref{geom_modular_fiber}}\label{sec.alt-modular}

As discussed in the introduction, if $Q$ is any parabolic subgroup stabilizing $I \in \ideals$, 
then we can consider the map 
$$\mu^{I,Q}: G \times ^Q I \to \overline{\co_I}$$ 
and its derived pushforward
$$\cs_{I,Q}:= R\mu^{I,Q}_*(\sh[\dim G/Q+\dim I]).$$
Suppose $Q'$ is another parabolic subgroup stabilizing $I$ with $Q' \subset Q$.
Then 
\begin{equation}\label{change_parabolics}
\cs_{I,Q'}= H^*(Q/Q')[\dim Q/Q'] \otimes \cs_{I,Q}
\end{equation}
since 
$G \times^{Q'} I$ is a fiber bundle over $G \times^Q I$ with fiber $Q/Q'$.

Now let $I_2\subset I_1\subset I_0$ be a modular triple of ideals.
Since the $f^I_{\varphi}$ satisfy $(1+q)f^{I_1}_\varphi =  f^{I_2}_\varphi + q f^{I_0}_\varphi$ by Proposition \ref{mod_laws},
the $m^I_{u,\phi}$ from \S\ref{sec.proof_thmA} satisfy 
$$(q+q^{-1})m^{I_1}_{u,\phi}=  m^{I_2}_{u,\phi} +  m^{I_0}_{u,\phi}$$ 
for all $(u,\phi) \in \sppairs$.
Since the $m^I_{u,\phi}$ determine the $V^I_{u,\phi}$ in Equation \eqref{decomp_springerized}, it follows that
$$(q+q^{-1})\cs_{I_1,B}  = \cs_{I_0,B} \oplus \cs_{I_2,B}. $$

Since $I_1$ and $I_0$ are $P_\alpha$-stable and 
$\mathsf P(P_\alpha/B) = 1+q$, we 
have
$\cs_{I_0,B}= (q+q^{-1})\cs_{I_0,P_\alpha}$ and
$\cs_{I_2,B}= (q+q^{-1})\cs_{I_2,P_\alpha}$.
Hence, clearing away the $(q+q^{-1})$ from all three terms, we have the following.

\begin{prop} \label{ThmB}\label{geom_modular_v2}
	Let $I_2\subset I_1\subset I_0$ be a modular triple.   There is an isomorphism
	$$\cs_{I_1,B} \simeq \cs_{I_2,P_\alpha} \oplus \cs_{I_0,P_\alpha} $$  
	in the derived category of $G$-equivariant perverse sheaves on $\cn$.
\end{prop}


\section{Type $A$ results} \label{sec.modular.revisited}   

In Proposition \ref{mod_laws} we proved that the polynomials 
$f_\varphi^I$ and $g_\varphi^{I^\perp}$ satisfy the modular law of \eqref{second_kind}.
We now connect those results to a combinatorial modular law for symmetric functions, proving that the two notions are equivalent in the type A case. 

Let $G=SL_n(\C)$ throughout this section, $B$ be the set of upper triangular matrices in $G$, and $T$ the set of diagonal matrices.  Let $E_{ij}$ denote the elementary matrix with $1$ in entry $(i,j)$ and all other entries equal to $0$.
Then the $E_{ij}$ for $i<j$ are basis vectors of the positive root spaces of $\fg=\mathfrak{sl}_n(\C)$ relative to $\ft$ and $\fb$.  
The positive roots are
$\epsilon_i - \epsilon_{j}$ for $i<j$, where $\epsilon_k$ denotes the linear dual of $E_{kk}$,
and the simple roots $\Delta$ are $\alpha_k := \epsilon_k - \epsilon_{k+1}$ for $1\leq k \leq n-1$. 
The simple reflection $s_k:= s_{\alpha_k}$ corresponds to the simple transposition in $W\simeq S_n$ exchanging $k$ and $k+1$.

The modular law was first introduced for chromatic symmetric functions by Guay--Paquet in~\cite{Guay-Paquet2013}.   More recently, Abreu and Nigro showed that any collection of multiplicative symmetric functions satisfying the modular law are uniquely determined up to some initial values~\cite{Abreu-Nigro2020}.   As an application of our results, we apply their theorem to compute the Frobenius characteristic of the dot action and LLT representations, recovering results of Brosnan--Chow and Guay--Paquet \cite{Brosnan-Chow2018, Guay-Paquet2016}.

\subsection{The combinatorial modular law}
We introduce the combinatorial modular law in the context of Hessenberg functions.  In type $A_{n-1}$, each ideal $I\in \ideals$
uniquely determines, and is determined by, 
a weakly increasing function 
$$h: \{1,2, \dots, n\} \to \{1,2, \dots, n\}$$
such that $i \leq h(i)$ for all $i$.
Such a function is called a Hessenberg function and its Hessenberg vector is $(h(1), \dots, h(n))$.
Given a Hessenberg function $h$, the ideal $I_h \in \ideals$ corresponding to $h$ is given by
\begin{eqnarray}\label{eqn.Hess-function}
I_h:= \mathrm{span}_{\C}\{ E_{ij} \mid  h(i)<j \}.
\end{eqnarray}

The function $h$ also determines a lattice path from the upper left corner of an $n\times n$ grid to the lower right corner, by requiring that the vertical step in row $i$ occurs $h(i)$ columns from the left. The requirement that $h(i)\geq i$ guarantees that this lattice path never crosses the diagonal.
Thus, Hessenberg functions (and ideals $\ideals$ and Hessenberg spaces $\ch$) are in bijection with the set of Dyck paths of length $2n$.  
By a slight abuse of notation,
we write $\ch$ for the set of Hessenberg functions, or equivalently Dyck paths, throughout this section.

\begin{example}\label{ex1.typeA} For $n=4$, the Hessenberg function $h$ with vector $(2,3,3,4)$ corresponds to the ideal $I_h = \mathrm{span}_\C\{E_{13}, E_{14}, E_{24}, E_{34}\}$ and defines the following lattice path.  The matrices in $I_h$ are those with all zeros below the path.  
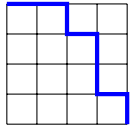
\begin{center}
\begin{tikzpicture}[scale=.4]
\draw (0,0) grid (4,4);
\draw[ultra thick,blue] 
(0,4) -- (2,4) -- (2,3) -- (3,3) -- (3,1)-- (4,1) -- (4,0);
\end{tikzpicture}
\end{center}
\end{example}

\begin{lemma} \label{lem.min-roots}
Let $h \in \ch$ and $\beta=\epsilon_i-\epsilon_j \in \Phi_{I_h}$.  Then $\beta$ is a minimal root of $\Phi_{I_h}$ if and only if $j=h(i)+1$ and $h(i)<h(i+1)$.
\end{lemma}
\begin{proof} Suppose $\beta =\epsilon_i - \epsilon_j \in \Phi_{I_h}$, i.e.~that $h(i)<j$.  Assume first that $j=i+1$.  In this case, the lemma is trivial since $\beta$ is a simple root (and thus a minimal root of $\Phi^+$) and $h(i)<i+1$ if and only if $h(i)=i$.  We may therefore assume $j>i+1$ for the remainder of the proof.  In this case, we have $\alpha\in \Delta$ such that $\beta-\alpha \in \Phi$ if and only if $\alpha = \alpha_i$ or $\alpha= \alpha_{j-1}$ since $\Phi$ is a type $A$ root system. Now $\beta\in \Phi_{I_h}$ is minimal if and only if 
\[
\beta -\alpha_{j-1} =\epsilon_i-\epsilon_{j-1} \notin \Phi_{I_h} \; \textup{ and } \; \beta - \alpha_{i} = \epsilon_{i+1}-\epsilon_j \notin \Phi_{I_h},
\]
or equivalently, $h(i) \geq j-1$ and $h(i+1)\geq j$. As $h(i)<j$ these conditions are equivalent to $h(i)=j-1$ and $h(i)<h(i+1)$, as desired.
\end{proof}

\begin{lemma}\label{lem.reflection-action} 
Let $h \in \ch$ and $k\in \{1, 2, \ldots, n-1\}$.  Then $s_k(\Phi_{I_h}) = \Phi_{I_h}$
if and only if $h(k)=h(k+1)$ and $h^{-1}(k)=\emptyset$.
\end{lemma}
\begin{proof} Suppose first that $h(k)=h(k+1)$ and $h^{-1}(k)=\emptyset$.  Note that $h(k)=h(k+1)\Rightarrow h(k)\geq k+1$ so $\alpha_k \notin \Phi_{I_h}$.  Let $\beta= \epsilon_i-\epsilon_j \in \Phi_{I_h}$, so $h(i)<j$. Consider $s_k(\beta) = \beta - \langle \beta, \alpha_k^\vee \rangle \alpha_k$.  If $\langle \beta, \alpha_k^\vee \rangle \leq 0$ then $s_k(\beta) \in \Phi_{I_h}$ since $I_h \in \ideals$.  We may therefore assume $\langle \beta, \alpha_k^\vee \rangle >0$, in which case we have either $i=k$ or $j=k+1$ since $\Phi$ is a type A root system and $\beta\neq \alpha_k$. If $i=k$ then $s_k(\beta) = \epsilon_{k+1}-\epsilon_j \in \Phi_{I_h}$ since $h(k+1) = h(i)<j$.  If $j=k+1$, then $s_k(\beta) = \epsilon_{i}-\epsilon_k \in \Phi_{I_h}$ since $h(i)<k+1$ and $h^{-1}(k)=\emptyset$ implies $h(i)<k$. 

Now suppose either $h(k)<h(k+1)$ or $h^{-1}(k)\neq \emptyset$.  If $h(k)<h(k+1)$, then $\epsilon_k- \epsilon_{h(k)+1}$ is a minimal root of $\Phi_{I_h}$ by Lemma~\ref{lem.min-roots} so 
\[
s_k(\epsilon_k - \epsilon_{h(k)+1}) = \epsilon_{k+1}- \epsilon_{h(k)+1} =  (\epsilon_k - \epsilon_{h(k)+1}) -\alpha_k \notin \Phi_{I_h},
\]
and $\Phi_{I_h}$ is not $s_k$-invariant.  If $h^{-1}(k)\neq \emptyset$ then there exists $i\in \{1,2, \ldots, n-1\}$ such that $h(i)=k$ and $h(i)<h(i+1)$.  We have that $\epsilon_i - \epsilon_{k+1}$ is a minimal root of $\Phi_{I_h}$ so 
\[
s_k(\epsilon_i - \epsilon_{k+1}) = \epsilon_{i}- \epsilon_{k} =  (\epsilon_i - \epsilon_{k+1}) -\alpha_k \notin \Phi_{I_h},
\]
and, as before, $\Phi_{I_h}$ is not $s_k$-invariant.  
\end{proof}

We now introduce the triples used to define the combinatorial modular law as in~\cite[Def.~2.1]{Abreu-Nigro2020}.  

\begin{definition} \label{def.modular-typeA} We say  $h_0,h_1, h_2 \in \ch$ is a \textbf{combinatorial triple} whenever one of the following two conditions holds:
\begin{enumerate}
\item There exists $i\in \{1, 2, \ldots, n-1\}$ such that $h_1(i-1)<h_1(i)<h_1(i+1)$ and $h_1(h_1(i)) = h_1(h_1(i)+1)$.
Moreover, $h_0$ and $h_2$ are defined to be: 
\begin{eqnarray}\label{eqn.h0h2cond1}
h_0(j)&=& \left\{\begin{array}{ll} h_1(j) & \textup{$j\neq i$}\\ h_1(i)-1 & \textup{$j=i$} \end{array}\right. \textup{ and }\\
 \nonumber h_2(j)&=& \left\{\begin{array}{ll} h_1(j) & \textup{$j\neq i$}\\ h_1(i)+1 & \textup{$j=i$.} \end{array}\right.
\end{eqnarray}
\item There exists $i\in \{1, 2, \ldots, n-1\}$ such that $h_1(i+1) = h_1(i)+1$ and $h_1^{-1}(i)=\emptyset$.  Moreover $h_0$ and $h_2$ are defined to be:
\begin{eqnarray}\label{eqn.h0h2cond2}
h_0(j)&=& \left\{\begin{array}{ll} h_1(j) & \textup{$j\neq i+1$}\\ h_1(i) & \textup{$j= i+1$} \end{array}\right. \textup{ and } \\
\nonumber h_2(j)&=& \left\{\begin{array}{ll} h_1(j) & \textup{$j\neq i $}\\ h_1(i+1) & \textup{$j=i$.} \end{array}\right.
\end{eqnarray}
\end{enumerate}
\end{definition}

\begin{example}\label{ex.typeA.1} Let $n=6$ and set $h_1=(2,3,4,6,6,6)$.  Then $h_1$ satisfies condition (1) of Definition~\ref{def.modular-typeA} for $i=3$ since $h_1(2)<h_1(3)<h_1(4)$ and
\[
h_1 (h_1(3)) = h_1(4) = 6 = h_1(4+1) = h_1(h_1(3)+1).
\]
Using~\eqref{eqn.h0h2cond1} we obtain $h_0=(2,3,3,6,6,6)$ and $h_2=(2,3,5,6,6,6)$.  
\end{example}

Given a triple $h_0, h_1, h_2$ of  Hessenberg functions we let $I_0, I_1, I_2\in \ideals$ be the corresponding ideals defined as in~\eqref{eqn.Hess-function}.  

\begin{lemma} \label{lemma.triples}
The Hessenberg functions $h_0, h_1, h_2$ form a combinatorial triple if and only if $I_2\subset I_1\subset I_0 $ is a modular triple of ideals.
\end{lemma}
\begin{proof} We first prove that any combinatorial triple corresponds to a
	modular triple of ideals in $\ideals$.  Let $h_0, h_1, h_2$ be a triple of Hessenberg functions satisfying condition (1) of Definition~\ref{def.modular-typeA}. We must have $h_1(i)>i$ in this case.  Indeed, $h_1(i)\geq i$ and if $h_1(i)=i$ then $h_1(h_1(i)) = i$ and $h_1(h_1(i)+1)= h_1(i+1)\geq i+1$, violating the fact that $h_1(h_1(i)) = h_1(h_1(i)+1)$.  We may therefore assume $i<h_1(i)<n$.  The formula for $h_0$ in~\eqref{eqn.h0h2cond1} now yields
\begin{eqnarray}\label{eqn.h0eq}
h_0(h_1(i)) = h_1(h_1(i)) = h_1(h_1(i)+1) = h_0(h_1(i)+1).
\end{eqnarray}
Similarly, if $j<i$ then $h_0(j) = h_1(j) \leq h_1(i-1)<h_1(i)$ and if $j>i$ then $h_0(j)=h_1(j)\geq h_1(i+1)>h_1(i)$. As $h_0(i) =h_1(i)-1 \neq h_1(i)$, this proves $h_0^{-1}(h_1(i)) = \emptyset$.  Together with~\eqref{eqn.h0eq}, this gives us $s_{h_1(i)}(I_{h_0}) = I_{h_0}$ by Lemma~\ref{lem.reflection-action}.  Set $\beta = \epsilon_i-\epsilon_{h_1(i)}$. Then $\langle \beta, \alpha_{h_1(i)}^\vee \rangle = -1$.  As $h_1(i) = h_0(i)+1$ and 
\[
h_0(i) = h_1(i)-1 < h_1(i+1) = h_0(i+1),
\]
we have that $\beta$ is a minimal root of $I_0$ by Lemma~\ref{lem.min-roots}. 
The formulas for $h_0$ and $h_2$ given in~\eqref{eqn.h0h2cond1} imply $I_0 = I_1\oplus \fg_{\beta}$ and $I_1=I_2\oplus \fg_{\alpha_{h_1(i)}+\beta}$. This proves $I_2\subset I_1\subset I_0$ is a modular triple.


Now suppose $h_0, h_1, h_2$ satisfy condition (2) of Definition~\ref{def.modular-typeA}. The formula for $h_0$ in~\eqref{eqn.h0h2cond2} gives us $h_1^{-1}(i)=\emptyset \Rightarrow h_0^{-1}(i)=\emptyset$ and $h_0(i)=h_1(i) = h_0(i+1)$.  Therefore $s_i(\Phi_{I_0}) = \Phi_{I_0}$ by Lemma~\ref{lem.reflection-action}. The assumption $h_1(i+1)=h_1(i)+1$ yields
\[
h_0(i+1)+1 = h_1(i)+1=h_1(i+1) 
\]
and 
\[
h_0(i+1) =h_1(i) < h_1(i+1) \leq h_1(i+2) = h_0(i+1)
\]
so $\epsilon_{i+1}-\epsilon_{h_1(i+1)}$ is a minimal root of $\Phi_{I_0}$ by Lemma~\ref{lem.min-roots}.  Setting $\beta=\epsilon_{i+1}-\epsilon_{h_1(i+1)}$ we get  $\left< \beta, \alpha_{i}^\vee \right> = -1$,  $I_0 = I_1\oplus \fg_{\beta}$, and $I_1=I_2\oplus \fg_{\alpha_i+\beta}$.  This proves $I_2\subset I_1\subset I_0$ is a modular triple.

Next we argue that any modular triple of ideals $I_2\subset I_1 \subset I_0$ corresponds to a combinatorial triple $h_0, h_1 ,  h_2$.  There exists $\alpha_k \in \Delta$ and $\beta\in \Phi^+$ such that $\left< \beta, \alpha_k^\vee \right> = -1$, $I_0 = I_1\oplus \fg_\beta$, and $s_k(\Phi_{I_0}) = \Phi_{I_0}$.  Furthermore, by definition $I_2$ is the ideal defined by the condition $I_1 = I_2\oplus \fg_{\alpha+\beta}$. Since $\Phi$ is a type A root system, $\beta=\epsilon_i-\epsilon_k$ for some $i<k$ or $\beta=  \epsilon_{k+1}-\epsilon_p$ for some $p>k+1$. We consider each case.


Suppose first that $\beta = \epsilon_i - \epsilon_k$ for some $i<k$.  We argue that the triple of Hessenberg functions $h_0 , h_1 , h_2$ satisfies condition (1) of Definition~\ref{def.modular-typeA}.
  As $\alpha +\beta= \epsilon_{i}-\epsilon_{k+1}$ is a minimal root of $I_1$, we have $h_1(i)=k$ and $h_1(i)<h_1(i+1)$ by Lemma~\ref{lem.min-roots}.  Since $I_0 = I_1\oplus \fg_\beta$,  
\begin{eqnarray}\label{eqn.h0}
h_0(j) = h_1(j) \; \textup{ for all } \; j\neq i \; \textup{ and } h_0(i) = h_1(i)-1.
\end{eqnarray}
If $h_1(i-1)=h_1(i)=k$ then $h_0(i-1) = k$ by~\eqref{eqn.h0}, contradicting, by Lemma~\ref{lem.reflection-action}, the assumption that $s_k(\Phi_{I_0}) = \Phi_{I_0}$.  Thus $h_1(i-1)<h_1(i)$.  Lemma~\ref{lem.reflection-action} also implies $h_0(k) = h_0(k+1)$ and since $i<k$, \eqref{eqn.h0} now yields $h_1(k) = h_1(k+1) \Rightarrow h_1(h_1(i)) = h_1(h_1(i)+1)$.
Finally, as $I_2$ satisfies $I_1=I_2\oplus \fg_{\alpha+\beta}$, the Hessenberg function $h_2$ is defined as in~\eqref{eqn.h0h2cond1}, concluding this case.

Suppose $\beta = \epsilon_{k+1}-\epsilon_p$ for some $p>k+1$.  We argue that the triple of Hessenberg functions $h_0 , h_1 , h_2$ satisfies condition (2) of Definition~\ref{def.modular-typeA}.
Lemma~\ref{lem.min-roots} implies that $p =h_0(k+1)+1$ since $\beta$ is a minimal root of $\Phi_{h_0}$. Now that fact that $I_0=I_1\oplus \fg_\beta$ implies 
\begin{eqnarray}\label{eqn.h1}
h_1(k+1) = p = h_0(k+1)+1 \; \textup{ and } \; h_1(j)=h_0(j) \; \textup{ for all }\;  j\neq k+1.
\end{eqnarray}
Lemma~\ref{lem.reflection-action} and~\eqref{eqn.h1} together imply that $h_1(k+1) = h_0(k)+1 = h_1(k)+1$ and $h_1^{-1}(k) = \emptyset$. This proves $h_1$ and $h_0$ are as in Definition~\ref{def.modular-typeA} with $i=k$.  The fact that $I_2$ satisfies $I_1=I_2\oplus \fg_{\alpha+\beta}$ where $\alpha+\beta = \epsilon_k - \epsilon_{p}$ implies that $h_2$ is defined as in~\eqref{eqn.h0h2cond2} with $i=k$.  The proof is now complete.
\end{proof}

Let $\Lambda^n$ denote the $\Z$-module of homogeneous symmetric functions of degree $n$.   
The Schur functions $\{s_\lambda \mid \lambda\vdash n\}$ are a basis of $\Lambda^n$.  Here, 
$\lambda\vdash n$ means $\lambda = (\lambda_1 \geq \lambda_2 \geq \cdots \geq \lambda_\ell>0)$ is 
a partition of $n$.

We say the function $F: \ch \to \Q(q)\otimes \Lambda^n$ satisfies the \textbf{combinatorial modular law} if 
\begin{eqnarray}\label{eq.modular-typeA}
(1+q)\,F(h_1) = F(h_2) + qF(h_0)
\end{eqnarray}
whenever  $h_0,h_1, h_2 \in \ch$ is a combinatorial triple of Hessenberg functions. 
With Lemma~\ref{lemma.triples} in hand, we recover the combinatorial modular law as a special case of the modular law from
Definition~\ref{def:modular_law_general}.  

Let $F:\ch \to \Q(q)\otimes \Lambda^n$.  Given $I \in \ideals$ and $\lambda\vdash n$, define $F_\lambda^{I}\in  \Q(q)$ by 
\[
F(h) = \sum_{\lambda\vdash n} F_\lambda^{I}(q) s_\lambda
\]
where $h \in \ch$ is the unique Hessenberg function with $I=I_h$.

\begin{prop}\label{prop.comb-to-geometric} The  function $F: \ch \to \Q(q)\otimes \Lambda^n$ satisfies the combinatorial modular law~\eqref{eq.modular-typeA} if and only if $\{F_\lambda^I\}_{I \in \ideals}$ satisfies~\eqref{second_kind} for all $\lambda \vdash n$.
\end{prop}
\begin{proof} By Lemma~\ref{lemma.triples}, $h_0, h_1, h_2\in \ch$ is a combinatorial triple if and only if $I_2\subset I_1\subset I_0$ is a modular triple in $\ideals$.  Since the $s_\lambda$ form a basis of $\Lambda^n$, it follows that
$$(1+q)\,F(h_1) = q\,F(h_0) + F(h_2)$$ 
if and only if 
$$(1+q)\,F_\lambda^{I_1} = qF_\lambda^{I_0} + F_\lambda^{I_2}$$ 
for all $\lambda\vdash n$. \end{proof}

\subsection{Chromatic quasisymmetric functions and LLT polynomials}

When working in the context of chromatic and LLT polynomials, a Hessenberg function (or Dyck path) is frequently identified with an indifference graph (see \cite[\S 2.3]{Harada-Precup2019}).  
For a composition $\mu = (\mu_1, \dots, \mu_r)$ of $n$,  let $h^{(\mu)}$ denote the Hessenberg function
with $h(i) = \mu_1 + \dots  + \mu_k$ where $k$ is the smallest index satisfying $i \leq \mu_1 + \dots  + \mu_k$.
 Let $K_m$ denote the complete graph on $m$ vertices.  
Then the graph corresponding to $h^{(\mu)}$ is a disjoint union $K_{\mu_1}  \sqcup K_{\mu_2} \sqcup \dots \sqcup K_{\mu_r}$ 
of complete graphs.  

Abreu-Nigro showed the following key result in \cite[Theorem 1.2]{Abreu-Nigro2020}.
\begin{thm}[Abreu--Nigro]\label{thm.Abreu-Nigro}  Let $F: \ch \to \Q(q)\otimes \Lambda^n$ be a function satisfying the combinatorial modular law of Equation~\eqref{eq.modular-typeA}. Then $F$ is determined by its values $F(h^{(\mu)})$ where $\mu$ is a composition of $n$.
\end{thm}

Let $\mathcal{R}_n$ denote the representation ring of $S_n$ and recall that the Frobenius characteristic map defines an isomorphism $\mathcal{R} = \oplus_n \mathcal{R}_n \to \Lambda=\oplus_n \Lambda^n$ (see~\cite[Section 7.3]{Fulton-YT}).  Given a graded complex vector space $U=\oplus_i U_{2i}$ concentrated in even degree such that each $U_{2i}$ is a finite dimensional $S_n$-representation, we let $\Ch(U) = \sum_i [U_{2i}]q^i$ where $[U_{2i}]$ denotes the class of $U_{2i}$ in~$\mathcal{R}_n$.

If $\mu$ is a composition of $n$ and $I\in \ideals$ is the ideal corresponding $h^{(\mu)}$, then   
$H = I^\perp \in \ch$ is a parabolic subalgebra $\fp_\mu$ of a parabolic subgroup $P_\mu$ in $G$
with $W_{P_\mu} \simeq S_{\mu_1} \times S_{\mu_2} \times \dots \times S_{\mu_r}$.
Since we know the values of the dot action and LLT representations at all $\fp_\mu$ by  Proposition~\ref{cor.LLT-parabolic}, we can use Theorem~\ref{thm.Abreu-Nigro} to identify the graded characters of the dot action and LLT representations as symmetric functions under the Frobenius characteristic map.  In the former case, we obtain another proof of the Shareshian--Wachs conjecture, originally proved by Brosnan and Chow in~\cite{Brosnan-Chow2018} and again using independent methods by Guay-Paquet in~\cite{Guay-Paquet2016}).  We note that the proof below is not wholly independent of that of Brosnan and Chow, since our computations in the previous section rely on their result \cite{Brosnan-Chow2018} identifying the dot action as a monodromy action. 

Rather than define the chromatic quasisymmetric and unicellular LLT functions corresponding to a given Hessenberg function $h$ in careful detail here, we refer the interested reader to~\cite{Shareshian-Wachs2016, Abreu-Nigro2020, Abreu-Nigro2, Alexandersson-Panova2018}.

\begin{cor} \label{cor.typeA-characters} Let $h\in \ch$ and let $H = I_h^\perp$.
\begin{enumerate}
\item The image of $\Ch(H^*(\hess{s}{H}) \otimes \mbox{sgn})$ under the Frobenius characteristic map is equal to the chromatic quasisymmetric function of $h$. 
\item The image of $\Ch(\LLT_{H})$ under the Frobenius characteristic map is equal to the unicellular LLT polynomial of $h$.
\end{enumerate}
\end{cor}

\begin{proof} 
Proposition~\ref{prop.comb-to-geometric} and  Proposition~\ref{mod_laws} imply that the Frobenius characteristic map of $\Ch(H^*(\hess{s}{H})\otimes \mbox{sgn})$ and $\Ch(\LLT_{H})$ in $\Q(q)\otimes \Lambda^n$ both satisfy the combinatorial modular law. 
It therefore suffices by Theorem~\ref{thm.Abreu-Nigro} to show that $\Ch(H^*(\hess{s}{\fp_\mu})\otimes \mbox{sgn})$ is equal to the chromatic quasisymmetric function of $h^{(\mu)}$ and $\Ch(\LLT_{\fp_\mu})$ is equal to the unicellular LLT polynomial of $h^{(\mu)}$ under the Frobenius characteristic map,  for each composition  $\mu$ of $n$.  

By Proposition~\ref{cor.LLT-parabolic}(1) 
$$\Ch(H^*(\hess{s}{\fp_\mu})\otimes \mbox{sgn}) =   \mathsf P(P_\mu/B) \Ch(\Ind^{S_n}_{S_{\mu_1} \times \dots \times S_{\mu_r}}(\mbox{sgn})).$$ 
The Frobenius characteristic of $\Ind^{S_n}_{S_{\mu_1} \times \dots \times S_{\mu_r}}(\mbox{sgn})$ is the elementary symmetric polynomial $e_\mu$.  Also $P(P_\mu/B) = \prod_{i=1}^r [\mu_i]!$ where $[m]! = [m][m-1]\dots[2][1]$ and $[m]$ is the $q$-number $1\!+\!q\! +\! \dots \!+\!q^{m-1}$.  The chromatic quasisymmetric function of $h^{(\mu)}$ is $\prod_{i=1}^r [\mu_i]! e_\mu$, so both sides agree in the parabolic setting and the result follows.

By Proposition~\ref{cor.LLT-parabolic}(2), 
$\Ch(\LLT_{\fp_\mu})$ 
equals the graded character of 
$$\Ind^{S_n}_{S_{\mu_1} \times \dots \times S_{\mu_r}}(\cc_{\mu_1} \otimes \cc_{\mu_2} \otimes \dots \otimes \cc_{\mu_r}),$$
where $\cc_{m}$ denotes the covariant algebra of $S_m$.
Let $\mbox{Frob}(U)$ denote the Frobenius characteristic of a graded representation $U$.  Then 
$$\mbox{Frob}\left(\Ind^{S_n}_{S_{\mu_1} \times \dots \times S_{\mu_r}}(\cc_{\mu_1} \otimes \cc_{\mu_2} \dots \otimes \cc_{\mu_r})\right) = \prod_{i=1}^r \mbox{Frob}(\cc_{\mu_i}).$$
Since unicellular LLT polynomials corresponding to Hessenberg functions are multiplicative (see~\cite[Theorem 2.4]{Abreu-Nigro2}), it suffices to know that 
$\mbox{Frob}(\cc_{m})$ equals the unicellular LLT polynomials for $S_m$ for the case of the complete graph $K_m$,
which is true by, for example, \cite[Equation (30)]{Alexandersson-Panova2018}.    
\end{proof}

\section{Example in $B_3$}\label{section:example}

In type $B_3$, there are $20$ ideals in $\ideals$ and $10$ irreducible representations in $\Wr$.  
The ideals are grouped according to the nilpotent orbit $\co_I$.   
We specify $I$  by listing the minimal roots in $\Phi_I$, using the coefficients of the simple roots. 
The elements of $\sppairs$ are listed in the top row of Table~\ref{fibers_B3}.
The elements of $\Wr$, as pairs of partitions, are listed in the top row of Table~\ref{table:key_poly}.

The ideals $I$ for which $\co_I = \co_x$ and 
$$I \subset \bigoplus_{i \geq 2} \fg_i$$ 
for the grading induced by $x$ are listed with a $*$ symbol.  For such $I$, we have $I \cap \fg_2 \in \ideals_{2}^{gen}$
and all such elements of $\ideals_{2}^{gen}$ arise in this way.

\begin{table}
	\resizebox{\linewidth}{!}{
	\noindent\begin{tabular}{|c|c|c|c|c|c|c|c|c|c|c|c|}
			\hline
			Ideal & Min roots&   $[1^7]$ & $[2^2,1^3]$  & $[3,1^4], \epsilon$ & $[3,1^4], 1$ & $[3,2^2]$ & $[3^2,1], \epsilon$ & $[3^2,1], 1$ & $[5,1^2]. \epsilon$ & $[5,1^2], 1$ & $[7]$ \\
			\hline
			\hline
			20* &$\emptyset$ & $[2][4][6]$ & $$ & $$ & $$ & $$ & $$ & $$ & $$ & $$ & $$  \\
			\hline
			19* & $122$ & $[2][4][6]$ & $[2][2]$ & $$ & $$ & $$ & $$ & $$ & $$ &  & \\
			18 & $112$ & $[2][4][6]$ & $[2][2][2]$ & $$ & $$ & $$ & $$ & $$ & $$ && \\
			17 & $012$ & $[2][4][6]$ & $[2][2][3]$ & $$ & $$ & $$ & $$ & $$ & $$  & &\\
			\hline
			16* & $111$ & $[2][4][6]$ & $[2][2][2]$ & $[2][2]$ & $[2][2]$ & $$ & $$ & &&&\\
			15* & $110$ & $[2][4][6]$ & $[2][2][3]$ & $q[2]$ & $[2][3]$ & $$ & $$ &  &&&\\
			14* & $100$ & $[2][4][6]$ & $[2][2][4] $ & $0$ & $[2][4]$ & $$ & $$ & &&& \\
			\hline
			13* & $111, 012$ & $[2][4][6]$ & $[2][2][3]$ & $[2][2]$ & $[2][2]$ & $[2]$ & $$ & $$ & $$ & $$ & \\
			12 & $011$ & $[2][4][6]$ & $[2][2][3]$ & $[2][2][2]$ & $[2][2][2]$ & $[2][2]$ & $$ & $$ & $$ & $$ & $$  \\
			11 & $001$ & $[2][4][6]$ & $[2][2][3]$ & $[2][2][3]$ & $[2][2][3]$ & $[2][3]$ & $$ & $$ & $$ & $$ & $$  \\
			\hline
			10* & $110,012$ & $[2][4][6]$ & $[2][2](1\!+q\!+2q^2)$ & $q[2]$ & $[2][3]$ & $[2]$ & $1$ & $1$ & $$ & &\\
			9 & $100,012$ & $[2][4][6]$ & $[2][2](1\!+\!q\!+\!2q^2\!+\!q^3)$ & $0$ & $[2][4]$ & $[2]$ & $[2]$ & $[2]$ & $$ & $$ & $$  \\
			8* & $110,011$ & $[2][4][6]$ & $[2][2](1\!+\!q\!+\!2q^2)$ & $2q\!+\!3q^2+q^3$ & $[2](1\!+\!2q\!+\!2q^2)$ & $[2](1\!+\!2q)$ & $0$ & $[2]$ & $$ & $$  &\\
			7* & $010$ & $[2][4][6]$ & $[2][2](1\!+\!q\!+\!2q^2\!+\!q^3)$ & $q[2][2]$ & $[2][2][3]$ & $[2][2][2]$ & $0$ & $[2][2]$ & $$ & $$ & $$ \\
			\hline
			6* & $100,011$ & $[2][4][6]$ & $[2][2](1\!+\!q\!+\!2q^2\!+\!q^3)$ & $q[2][2]$ & $[2][2][3]$ & $[2](1\!+\!2q)$ & $q $ & $1\!+\!2q$ & $1$ & $1$ &  \\
			5 & $110,001$ & $[2][4][6]$ & $[2][2](1\!+\!q\!+\!2q^2)$ & $q[2](2\!+\!2q\!+\!q^2)$ &   $[2](1\!+\!2q\!+\!3q^2\!+\!q^3)$ 
			& $[2][2][2]$ & $0$ & $[2]$ & $1$ & $1$ &   \\
			4 & $100,001$ & $[2][4][6]$ & $[2][2](1\!+\!q\!+\!2q^2\!+\!q^3)$ & $q[2][2][2]$ & $[2](1\!+\!2q\!+\!3q^2\!+\!2q^3)$ & $[2](1\!+\!2q\!+\!2q^2)$ & $0$ & $[2][2]$ & $[2]$ & $[2]$ & \\
			3 & $010,001$ & $[2][4][6]$ & $[2][2](1\!+\!q\!+\!2q^2\!+\!q^3)$ & $q[2][2][2]$ & $[2](1\!+\!2q\!+\!3q^2\!+\!2q^3)$ & $[2](1\!+\!2q\!+\!2q^2)$ & $0$ & $[2][2]$ & $[2]$ & $[2]$ &  \\
			2* & $100,010$ & $[2][4][6]$ & $[2][2](1\!+\!q\!+\!2q^2\!+\!2q^3)$ & $q^2[2]$ & $[2](1\!+\!2q\!+\!2q^2\!+\!2q^3)$ & $[2][2][2]$ & $q[2]$ & $1\!+\!3q\!+\!2q^2$ & $0$ & $[2]$ &  \\
			\hline
			1* & $\Delta$ & $[2][4][6]$ & $[2][2](1\!+\!q\!+\!2q^2\!+\!2q^3)$ & $2q^2\!+\!3q^3\!+\!q^4$ & $1\!+\!3q\!+\!5q^2\!+\!6q^3\!+\!3q^4$ & $1\!+\!3q\!+\!5q^2\!+\!3q^3$ & $q^2$ & $1\!+\!3q\!+\!3q^2$ & $2q$ & $1\!+\!3q$ & $1$ \\			
			\hline
		\end{tabular}}
	\caption{The polynomials $\mathsf P(\hess{x}{I};\chi)$ for $B_3$} 
		\label{fibers_B3}
\end{table}

\begin{table}
	\resizebox{\linewidth}{!}{
\begin{tabular}{|c|c|c|c|c|c|c|c|c|c|c|c|}
			
			\hline
			
			Ideal & Min roots & $\emptyset, [1^3]$ & $\emptyset, [2,1]$ & $[1^3], \emptyset$  & $[1], [1^2]$ & $[1^2], [1]$ & $\emptyset, [3]$ & $[1],[2]$ &  $[2,1],\emptyset$ & $[2],[1]$ & $[3], \emptyset$  \\
			\hline
			20 & $\emptyset$ &  $[2][4][6]$ & $$ & $$ & $$ & $$ & $$ & $$ & $$ & $$ & $$  \\
			\hline
			19 & $122$ & $[2][4][5]$ & $q^3[2][2]$ & $$ & $$ & $$ & $$ & $$ & $$ &  & \\
			18 & $112$ & $[2][4][4]$ & $q^2[2][2][2]$ & $$ & $$ & $$ & $$ & $$ & $$ && \\
			17 & $012$ & $[2][3][4]$ & $q[2][2][3]$ &  &  & $$ & $$ & $$ & $$  && \\
			\hline
			16 & $111$ &  $[2][3][4]$ & $q^2[2][2]$ & $q^2[2][2]$ & $q^2[2][2]$ & $$ & $$  && &&\\
			15 & $110$ &  $[2][2][4]$ & $q[2][3]$ & $q^2[2]$ & $q[2][3]$ & $$ & $$ &  && &\\
			14 & $100$ &$[2][4]$ & $[2][4]$ & $0$ & $[2][4]$ & $$ & $$ & &&  &\\
			\hline
			13 & $111, 012$ &  $[2][3][3]$ & $q[2][2][2]$ & $q^2[2]$ & $q^2[2]$ & $q^2[2]$ & $$ & $$ & $$ &  &  \\
			12 & $011$ & $[2][2][3]$ & $q[2][2]$ & $q[2][2]$ & $q[2][2]$ & $q[2][2]$ & $$ & $$ & $$ & $$ & $$  \\
			11 & $001$ &  $[2][3]$ & $0$ & $[2][3]$ & $[2][3]$ & $[2][3]$ & $$ & $$ & $$ & $$ & $$  \\
			\hline
			10 & $110,012$ & $[2][2][3]$ & $2q[2][2]$ & $q^2$ & $q[2][2]$ & $q^2$ & $q^2$ & $q^2$  &  &  &\\
			9 & $100,012$ & $[2][3]$ & $[2][2][2]$ & $0$ & $[2][3]$ & $0$ & $q[2]$ & $q[2]$ & $$ & $$ & $$  \\
			8 & $110,011$ & $[2][2][2]$ & $2q[2]$ & $q[2]$ & $2q[2]$ & $2q[2]$ & $0$ & $q[2]$ & $$ &  & \\
			7 & $010$ & $[2][2]$ & $[2][2]$ & $0$ & $[2][2]$ & $[2][2]$ & $0$ & $[2][2]$ & $$ & $$ &  \\
			\hline
			6 & $100,011$ & $[2][2]$ & $1\!+\!3q\!+\!q^2$ & $q$ & $1\!+\!3q\!+\!q^2$ & $2q$ & $q$ & $2q$ & $q$ & $q$ &   \\
			5 & $110,001$ & $[2][2]$ & $q$ & $[2][2]$ & $1\!+\!3q\!+\!q^2$ & $1\!+\!3q\!+\!q^2$ & $0$ & $q$ & $q$ & $q$ & \\
			4 & $100,001$ &  $[2]$ & $[2]$ & $[2]$ & $2[2]$ & $2[2]$ & $0$ & $[2]$ & $[2]$ & $[2]$  & \\
			3 & $010,001$ &  $[2]$ & $[2]$ & $[2]$ & $2[2]$ & $2[2]$ & $0$ & $[2]$ & $[2]$ & $[2]$ & \\
			2 & $100,010$ & $[2]$ & $2[2]$ & $0$ & $2[2]$ & $[2]$ & $[2]$ & $2[2]$ & $0$ & $[2]$ &  \\
			\hline
			1& $\Delta$ & $1$ & $2$ & $1$ & $3$ & $3$ & $1$ & $2$ & $2$ & $3$ & $1$ \\
			\hline
		\end{tabular}}
\caption{The polynomials $f^I_\varphi$ for $B_3$} 	\label{table:key_poly}	
\end{table}

There are $10$ examples of the modular law in type $B_3$, which we list using the numbering of the ideals in the tables:
\begin{eqnarray}\label{eqn:B3triples}
&&(3,5,11), (4,5,11 ),(4,6,9 ),  (7,8,12 ), (9,10,13 ), 	(11,12,13),\\ 
\nonumber &&\quad\quad(12,13,17),    (14,15,16),(17,18,19),(18,19,20).
\end{eqnarray}
The modular triples~\eqref{eqn:B3triples} provide a check on our calculation in both tables, as do the $8$ cases in the parabolic setting in Table \ref{table:key_poly}
using \eqref{parabolic_case}.  
We made use of Binegar's tables to check our calculation of the Green polynomial matrix ${\bf K}$
\cite{Binegar2022}.

\appendix

\section{Definition of the dot action and LLT representations}\label{appendix}
This appendix introduces the dot action on the equivariant cohomology of a regular semisimple Hessenberg variety and uses the construction to  define the dot action and LLT representations.  Our main goal is to obtain a proof of Proposition~\ref{prop.LLT} above, which was used to define and study the LLT representations. 
 
Let $H\in \ch$ be a Hessenberg space and $s\in \ft$ a regular semisimple element.  Recall that the torus $T$ acts on regular semisimple Hessenberg variety $\hess{s}{H}$ by left multiplication.
The variety is in fact equivariantly formal with respect to this action~\cite{Tymoczko2005}.  Applying the theory developed by Goresky, Kottwitz, and MacPherson  \cite{GKM1998}, the equivariant cohomology $H_T^*(\hess{s}{H})$ has the following description.  
The inclusion map of the $T$-fixed point $\hess{s}{H,T} = \{ wB \mid w\in W\}$ into $\hess{s}{H}$ induces an injection 
\[
\iota: H_T^*(\hess{s}{H}) \hookrightarrow H_T^*(\hess{s}{H,T}) \simeq \bigoplus_{w\in W} \C[\ft^*].
\]
Here $\C[\ft^*]$ denotes the polynomial ring in the simple roots $\C[\alpha_1, \ldots, \alpha_n]$. We identify $H_T^*(\hess{s}{H})$ with its image under this map, which is given by the following concrete description,
\begin{eqnarray}\label{eqn.equiv-cohomology}
\quad H_T^*(\hess{s}{H})  \simeq  \left\{ (f_w)_{w\in W} \mid  \begin{array}{c}f_w-f_{s_\gamma w}\in \left< \gamma \right> \textup{ for } \gamma\in \Phi^+,\\ w^{-1}(\gamma)\in \Phi_H\cap \Phi^- \end{array}  \right\}.
\end{eqnarray}
A more leisurely exposition of the above can be found in~\cite{AHMMS2020, Tymoczko2005}.

Let $\ct: = \bigoplus_{w\in W} \C[\ft^*]$.  The ring $\ct$ is a $\C[\ft^*]$-module via the action, 
\begin{eqnarray}\label{eqn.left-module}
(p,f) \mapsto pf \; \textup{ where } \; (pf)_w = pf_w 
\end{eqnarray}
for all $p\in \C[\ft^*]$ and $ f=(f_w)_{w\in W} \in \ct$.
Equation~\eqref{eqn.equiv-cohomology} identifies $H_T^*(\hess{s}{H})$ as a $\C[\ft^*]$-submodule of $\ct$, and we make this identification from now on.   
We can also view $\ct$ as a $\C[\ft^*]$-module in another way.  To distinguish this structure from that defined in~\eqref{eqn.left-module} we call this the right $\C[\ft^*]$-module action on $\ct$, which is defined by
\begin{eqnarray}\label{eqn.right-module}
(q, f )\mapsto fq \; \textup{ where } \; (fq)_w = w(q)f_w
\end{eqnarray}
for all $q\in \C[\ft^*]$ and $f=(f_w)_{w\in W} \in \ct$.
The equivariant cohomology $H_T^*(\hess{s}{H})$ is also a $\C[\ft^*]$-submodule of $\ct$ with respect to the right action.

Let $\mathbf{1} = (1)_{w\in W}\in \ct$.  We can identify $\C[\ft^*]$ with the $\C[\ft^*]$-submodule generated by $\mathbf{1}$, via the left action from~\eqref{eqn.left-module} or the right action from~\eqref{eqn.right-module}.  
To distinguish between the left and right submodules generated by $\mathbf{1}$ we write:
\begin{itemize}
\item $\C[L]$ for the left $\C[\ft^*]$-module generated by $\mathbf{1}$ via the action from~\eqref{eqn.left-module}.
\item $\C[R]$ for the right $\C[\ft^*]$-module generated by $\mathbf{1}$ via the action from~\eqref{eqn.right-module}.
\end{itemize}
This notation is inspired by the exposition in~\cite{Guay-Paquet2016}.  Both $\C[L]$ and $\C[R]$ are submodules of $H_T^*(\hess{s}{H})$ in the appropriate sense. Recall that $H^*(\hess{s}{H}) \simeq H_T^*(\hess{s}{H})/ \left<  \alpha_1, \ldots, \alpha_n\right>_L$ where $\left<  \alpha_1, \ldots, \alpha_n\right>_L$ is the ideal of $H_T^*(\hess{s}{H})$ generated by the positive degree elements in $\C[L]$ (see \cite[Prop.~2.3]{Tymoczko2005}).

The $W$-action on $\Phi \subseteq \ft^*$ extends to a $W$-action on $\C[\ft^*]$ in a natural way.  We denote the action of $w\in W$ on $f\in \C[\ft^*]$ by $w(f)$.  The dot action of $W$ on $\ct$ is defined by
\begin{eqnarray}\label{eqn.dot-action}
(v\cdot f)_w := v(f_{v^{-1}w}) \; \textup{ for all } \; v\in W,\, f\in \ct.
\end{eqnarray}
The dot action preserves the equivariant cohomology $H_T^*(\hess{s}{H})$ in $\ct$ \cite[Lemma 8.7]{AHMMS2020}.

The submodules $\C[L]$ and $\C[R]$ introduced above are also invariant under the dot action.  Let $\C[\ft^*]^W$ denote the ring of $W$-invariants in $\C[\ft^*]$.  By Chevalley's Theorem 
$\C[\ft^*] \simeq \C[\ft^*]^W\otimes \cc$ where $\cc$ is the coinvariant algebra of $W$.  The following lemma computes the graded character of the dot action on $\C[L]$ and $\C[R]$.

\begin{lemma}\label{lemma.char}  We have $\C[L]\simeq \C[\ft^*]^W\otimes \cc$ and $\C[R] \simeq \C[\ft^*]^W \otimes \cc'$, where $\cc'\simeq \cc$ as vector spaces, but $W$ acts trivially on $\cc'$.  In particular,
\[
\Ch(\C[L]) = \prod_{i=1}^{n} \frac{1}{(1-q^{d_i})}\,\Ch(\cc)
\]
and 
\[ \Ch(\C[R]) = \prod_{i=1}^{n} \frac{1}{(1-q^{d_i})} \, \sum_{w\in W} [1_W]q^{\ell(w)}
\]
where $d_1, \ldots, d_n$ denote the degrees of $W$ \textup{(}cf.~\cite[Sec.~3.7-3.8]{Humphreys-Coxeter}\textup{)}.
\end{lemma}
\begin{proof} We first compute the dot action on $\C[L]$ and $\C[R]$, respectively. If $p\in \C[L]$ then
\[
(v\cdot p)_w = v(p_{v^{-1}w}) = v(p) =  (v(p))_w.
\]
Thus the dot action on $\C[L]$ is the usual graded representation of $W$ on the polynomial ring $\C[\ft^*]$. 
If $q\in \C[R]$ then
\[
(v\cdot q)_w=v(q_{v^{-1}w}) = v(v^{-1}w(q)) = w(q)=q_w
\]
so the dot action on $\C[R]$ is trivial.  Since $\C[L]\simeq \C[R]\simeq \C[\ft^*]$ as vector spaces, the first assertion of the lemma follows from the above computations. Finally, the second assertion follows from the fact that the Poincar\'e polynomial of the ring $\C[\ft^*]^W$ is precisely $\prod_{i=1}^n \frac{1}{(1-q^{d_i})}$ and $\Ch(\cc') = \sum_{w\in W} [1_W]q^{\ell(w)}$.
\end{proof}

Consider the ideals $\left< \alpha_1, \ldots, \alpha_n \right>_L$, and respectively $\left< \alpha_1, \ldots, \alpha_n \right>_R$, in the equivariant cohomology $H_T^*(\hess{s}{H})$ generated by the positive degree elements of $\C[L]$, and respectively $\C[R]$. Both are $W$-invariant, and the dot action on $H_T^*(\hess{s}{H})$ induces an action on the quotients
\begin{eqnarray}\label{eqn.dot-action-rep}
H^*(\hess{s}{H}) \simeq H_T^*(\hess{s}{H})/\left<  \alpha_1, \ldots, \alpha_n\right>_L
\end{eqnarray}
and
\begin{eqnarray}\label{eqn.LLT-rep}
\LLT_H:= H_T^*(\hess{s}{H}) / \left< \alpha_1, \ldots, \alpha_n \right>_R.
\end{eqnarray}
We refer to the $W$-module $H^*(\hess{s}{H})$ as the dot action representation and the $W$-module $\LLT_H$ as the LLT representation.

\begin{Rem}
In the type A case $\LLT_H\simeq H^*(\cx_H)$ where $\cx_H$ is the smooth manifold of Hermitian matrices having a particular staircase form (determined by $H$) and a given fixed simple spectrum (determined by $s$) \cite{Ayzenberg-Buchstaber}.
\end{Rem}

We can now prove Proposition~\ref{prop.LLT}; our argument closely follows that of Guay-Paquet in~\cite{Guay-Paquet2016} for the type A case. In the proof below, $\star$ denotes the product (which corresponds to taking the tensor product of representations) in the character ring of $W$.

\begin{proof}[Proof of Proposition~\ref{prop.LLT}] The equivariant cohomology $H_T^*(\hess{s}{H})$ is a free module of rank $n!$ over $\C[\ft^*]$ with respect to either~\eqref{eqn.left-module} or~\eqref{eqn.right-module}, see the discussion in \cite[Section 8.5]{Guay-Paquet2016} or \cite[Section 2.3]{AHMMS2020}.  In particular, the usual extension of scalars construction for free modules yields isomorphisms of $W$-modules:
\begin{eqnarray}\label{eqn.scalars1}
H_T^*(\hess{s}{H}) \simeq \C[L]\otimes_\C H^*(\hess{s}{H})
\end{eqnarray}
and
\begin{eqnarray}\label{eqn.scalars2}
H_T^*(\hess{s}{H}) \simeq \C[R]\otimes_\C \LLT_H.
\end{eqnarray}
It now follows that 
\[
\Ch(\C[L]) \star \Ch(H^*(\hess{s}{H})) = \Ch(\C[R])\star \Ch(\LLT_H).
\]
Applying the formulas from Lemma~\ref{lemma.char} and dividing by $\prod_{i=1}^n \frac{1}{(1-q^{d_i})}$ now yields the desired result.
\end{proof}

To conclude, we sketch a proof of the fact that if $\hess{s}{H}$ is disconnected then both the dot action and LLT representations are induced by corresponding representations of a parabolic subgroup of $W$.  It is well known to experts but, to the best of our knowledge, has not appeared in the literature so we include an outline of the argument here.   
This fact can be used together with Proposition~\ref{ss_hessy} to give another proof of Borho and MacPherson's result in equation~\eqref{parabolic_case} above.

Let $J = \{ \alpha\in \Delta \mid -\alpha\in \Phi_H \} \subseteq \Delta$.  
Then $\hess{s}{H}$ is connected if and only if $J=\Delta$ \cite[Appendix A]{Anderson-Tymoczko2010}.  
Let $L$ denote the standard Levi subgroup associated to $J$ with Lie algebra $\fl$. There is a natural embedding of the flag variety $\cb_L:=L/(B\cap L)$ of $L$ into $\cb$ given by $\ell(B\cap L)\mapsto \ell B$.  Let $W_J:= \left< s_\alpha \mid \alpha\in J \right>$ be the corresponding parabolic subgroup of $W$, which is the Weyl group of $L$.  Denote by $W^J$ the set of shortest coset representatives for the left cosets $W/W_J$. For each $v\in W^J$ we have that $s_v:=v^{-1}.s$ is a regular semisimple element in $\fl$.  Furthermore, by definition $\bar H:=H\cap \fl$ is a Hessenberg space in $\fl$ and the regular semisimple Hessenberg variety $\hess{L,s}{\bar H}$ in the flag variety of $L$ is connected. Now the decomposition of $\hess{s}{H}$ into connected components is given by
\begin{eqnarray}\label{eqn.decomp}
\hess{s}{H} = \bigsqcup_{v\in W^J} v(\cb_{L,s_v}^{\bar H})
\end{eqnarray}
where $\cb_{L,s_v}^{\bar H} = \{\ell(B\cap L)\in \cb_L \mid \ell^{-1}.s_v\in \bar H \}$.
Each $v(\cb_{L,s_v}^{\bar H})$ is isomorphic to $\hess{L,s}{\bar H}$, and the cohomology decomposes accordingly.  
We now obtain the following directly from the definition of the dot action together with the decomposition cohomology induced by~\eqref{eqn.decomp}.

\begin{cor}\label{cor.induction} There is an isomorphism of $W$-modules:
\begin{eqnarray}\label{eqn.induction}
H_T^*(\hess{s}{H}) \simeq \Ind_{W_J}^W (H_T^*(\hess{L,s}{\bar H})).
\end{eqnarray}
In particular, both the dot action and LLT representations are obtained by induction from the corresponding representations for $\hess{L,s}{\bar H}$, namely,
\[
H^*(\hess{s}{H}) \simeq \Ind_{W_J}^W (H^*(\hess{L,s}{\bar H})) \; \textup{ and } \;  \LLT_H \simeq \Ind_{W_J}^W (\LLT_{\bar{H}}).
\]
\end{cor}


%
%

\newcommand{\etalchar}[1]{$^{#1}$}

\end{document}